\documentclass[10pt,reqno]{amsart}
\RequirePackage[OT1]{fontenc}
\RequirePackage{amsthm,amsmath}
\RequirePackage[numbers]{natbib}
\RequirePackage[colorlinks,linkcolor=blue,citecolor=blue,urlcolor=blue]{hyperref}
\usepackage{amssymb}
\usepackage{anysize}
\usepackage{hyperref}
\usepackage{extarrows}
\usepackage{multirow}
\usepackage{natbib}
\usepackage{stmaryrd}
\usepackage{color}
\usepackage{amsmath}
\usepackage{enumitem}
\usepackage{caption}
\usepackage{subcaption}  
 \usepackage{floatrow}
\usepackage{tikz}
\usetikzlibrary{calc,intersections,through}
\tikzset{point/.style={circle,inner sep=0pt,minimum size=3pt,fill=black}}
\usetikzlibrary{matrix, calc, arrows}
\usetikzlibrary{arrows.meta}

\definecolor{darkred}{rgb}{0.9,0,0.3}
\definecolor{darkblue}{rgb}{0,0.3,0.9}
\definecolor{test}{rgb}{1,0,1}

\newcommand{\black}{\normalcolor}

\numberwithin{equation}{section}

\theoremstyle{plain} 
\newtheorem{thm}{Theorem}[section]
\newtheorem{lem}[thm]{Lemma}

\newtheorem{pro}[thm]{Proposition}

\newtheorem{defn}[thm]{Definition}
\newtheorem{con}[thm]{Conjecture}

\theoremstyle{remark}
\newtheorem{rem}[thm]{Remark}

\renewcommand{\Im}{\mathrm{Im}\,}

\newcommand{\im}{\mathrm{Im}\,}

\newcommand{\E}{{\mathbb E }}
\newcommand{\R}{{\mathbb R }}
\newcommand{\N}{{\mathbb N}}

\newcommand{\C}{{\mathbb C}}

\newcommand{\ii}{\mathrm{i}}
\newcommand{\deq}{\mathrel{\mathop:}=}
\newcommand{\eqd}{\mathrel={\mathop:}}
\newcommand{\e}[1]{\mathrm{e}^{#1}}

\newcommand{\dd}{\mathrm{d}}

\newcommand{\bs}{\boldsymbol}

\newcommand{\la}{\langle}
\newcommand{\ra}{\rangle}
\newcommand{\qq}[1]{[\![{#1}]\!]}

\renewcommand{\mathbf}[1]{\bs{#1}}
 
\marginsize{24mm}{24mm}{24mm}{24mm}  

\usepackage{amsmath}
\usepackage{amsfonts}
\usepackage{latexsym}
\usepackage{amsthm}
\usepackage{amsxtra}
\usepackage{amscd}
\usepackage{bbm}
\usepackage{mathrsfs}
\usepackage{bm}
\usepackage{tensor}

\renewcommand{\leq}{\leqslant}
\renewcommand{\le}{\leqslant}
\renewcommand{\geq}{\geqslant}

\renewcommand{\epsilon}{\varepsilon}

\renewcommand{\b}[1]{\boldsymbol{\mathrm{#1}}} 
\newcommand{\bb}{\mathbb} 
\newcommand{\cal}{\mathcal}

\newcommand{\ul}[1]{\underline{#1} \!\,} 

\DeclareMathOperator{\cov}{Cov}
\DeclareMathOperator{\tr}{Tr}
\DeclareMathOperator{\var}{Var}

\DeclareMathOperator{\dist}{dist}

\newcommand{\floor}[1] {\lfloor {#1} \rfloor}

\begin{document}

\begin{center}
\large\bf
On Cram\'{e}r-von Mises statistic for the spectral distribution of random matrices
\end{center}

\vspace{2ex}
\renewcommand{\thefootnote}{\fnsymbol{footnote}}
\begin{center}
 \begin{minipage}[t]{0.35\textwidth}
\begin{center}
Zhigang Bao\footnotemark[1]  \\
\footnotesize {HKUST}\\
{\it mazgbao@ust.hk}
\end{center}
\end{minipage}
\hspace{8ex}
\begin{minipage}[t]{0.35\textwidth}
\begin{center}
Yukun He\footnotemark[2]  \\ 
\footnotesize {University of Z\"{u}rich}\\
{\it yukun.he@math.uzh.ch}
\end{center}
\end{minipage}

\footnotetext[1]{Supported by Hong Kong RGC grant ECS 26301517, GRF 16300618, and NSFC 11871425.}
\footnotetext[2]{Supported by NCCR Swissmap, SNF grant No.\ 20020\_1726.}

\renewcommand{\thefootnote}{\fnsymbol{footnote}}	

\end{center}

\vspace{0.5cm}
\begin{center}
 \begin{minipage}{0.8\textwidth}\footnotesize{
Let $F_N$ and $F$ be the empirical and limiting spectral distributions of an $N\times N$ Wigner matrix. The Cram\'{e}r-von Mises (CvM) statistic  is a classical goodness-of-fit statistic that characterizes the distance between $F_N$ and $F$ in $L^2$-norm. In this paper, we consider a mesoscopic approximation  of the
CvM statistic for Wigner matrices, and derive its limiting distribution. In the appendix, we also give the limiting distribution of the  CvM statistic (without approximation) for the toy model CUE.} 
\end{minipage}
\end{center}

\thispagestyle{headings}

\section{Introduction and main result}
\subsection{Background} \label{s.background}
Let $H=(H_{ij})_{N,N}$ be a Wigner matrix, i.e., $H_{ij}=H_{ji}^*$ and $H_{ij}$'s are independent (up to symmetry).  Further,  we assume 
\begin{enumerate}
	\item $\mathbb{E}H_{ij}=0$ for all $i,j$.
	\item $\mathbb{E} |H_{ij}|^2=\frac{1}{4N}$ for $i \ne j$, and $\bb E|H_{ii}|^2=\frac{\sigma^2}{4N}$ for all $i$.
	\item $\bb E|H_{ij}|^a=O(N^{-a/2})$ for all $i,j$ and all fixed positive integer $a$.
\end{enumerate}
We distinguish the \emph{real symmetric} case ($\beta=1$) where $H_{ij} \in \R$ for all $i,j$ from the \emph{complex Hermitian} case ($\beta=2$) where $\E H_{ij}^2 = 0$ for $i \neq j$. We further assume the fourth moments of $H_{ij}$ are homogeneous, i.e. $\bb E|H_{ij}|^4=m_4/N^2$ for all $i \ne j$, and denote $c_4=m_4-(4-\beta)/16$.  Let
\begin{align*}
\lambda_1\geq \lambda_2\geq \ldots \geq \lambda_N
\end{align*}
be the ordered eigenvalues of $H$.  Denote the empirical spectral distribution of $H$ by 
\begin{align*}
F_N(x):=\frac{1}{N}\sum_{i=1}^N \mathbf{1}(\lambda_i\leq x). 
\end{align*}
Arguably, the most fundamental result in Random matrix Theory (RMT) is Wigner's semicircle law \cite{Wigner}. It states that  almost surely $F_N(x)$ converges weakly to the semicircle law $F(x)$ with the density function given by 
\begin{align*}
\rho_{sc}(x)=\frac{2}{\pi} \sqrt{(1-x^2)_+}. 
\end{align*}
Based on this fundamental weak convergence result, some  natural questions can be asked further. For instance, one can ask   how to characterize the distance between $F_N$ and $F$. In addition, if a specific distance is chosen, we can take one more step to ask:   What is the limiting behaviour  of this distance between $F_N$ and $F$? The first question has a rich answer, considering that there are many widely used statistical distances for distributions in the literature. However, it could be very challenging to answer the second question for some specific distances, especially when one aims at some fine result like the limiting distributions of these distances between $F_N$ and $F$. 

In applied probability and statistics theory, two widely used  distances (or statistics) between the empirical distribution and the limiting one, are the  Cram\'{e}r-von Mises (CvM) statistic and the Kolmogorov-Smirnov (KS) statistic, which are $L^2$ and $L^\infty$ distances, respectively. More specifically, the CvM statistic for the spectral distribution of  $H$ is defined as 
\begin{align} 
\mathcal{A}_N:= \int (F_N(t)-F(t))^2 {\rm d} F(t),  \label{190529100}
\end{align}
and the   well-known KS statistic is defined by 
\begin{align*}
\mathcal{K}_N:= \max_{t}|F_N(t)-F(t)|.  
\end{align*}
We remark here that most of the literature on these two statistics are about the empirical distribution of the i.i.d. samples of a  random variable. Here we consider the same distances for the empirical distribution of the highly correlated eigenvalues of random matrices. In the statistics literature, both $\mathcal{A}_N$ and $\mathcal{K}_N$ belong to the category of  the {\it goodness-of-fit} statistics, which can be used to test how well the empirical distribution fits a given distribution. 
Specifically, both two statistics are fundamental for the following hypothesis testing problem: test if or not a random variable $X$ follows a given distribution ($F$), based on the empirical distribution ($F_N$) of its i.i.d. samples; see \cite{AD} for instance. Here we use $F_N$ and $F$ to denote the empirical and limiting distribution for the i.i.d. samples as well, with certain abuse of notations. For the i.i.d. setting, identifying the  limiting distributions of $\mathcal{A}_N$ and $\mathcal{K}_N$ and their variants had been the primary task for this topic, due to their  importance for the hypothesis testing problems. Nowadays, in the classical i.i.d. setting, under the null hypothesis,  it is well-known that the limiting distributions of both $\mathcal{A}_N$ and $\mathcal{K}_N$  follow from the classical Donsker's theorem. Specifically,  since $\sqrt{N}(F_N(t)-F(t))$ converges to the Brownian bridge, $\mathcal{A}_N$ and $\mathcal{K}_N$ converge to the corresponding functionals of Brownian bridge, after appropriate normalization. In the classical i.i.d. case,  the weight ${\rm d} F(t)$  in (\ref{190529100}) is chosen to make the statistic {\it nonparametric}.  That means, after a simple change of variable in (\ref{190529100}), no matter the distribution of the i.i.d. sample, it is always the same as the case of i.i.d. uniformly distributed samples. But in the literature, there are indeed many other choices of the weights. We refer to the reference \cite{AD, Stephens76}. In general, for the i.i.d. case, the distribution of the CvM type statistics with general weight is given by an infinite  sum of independent weighted $\chi^2$ random variables. The characteristic function of this  distribution can be written as a Fredholm determinant; see \cite{AD, Stephens76} for instance. 

To the best of our knowledge, in the context of RMT, the limiting distribution of $\mathcal{A}_N$ and $\mathcal{K}_N$ for the empirical spectral distribution have never been obtained. Nevertheless, many recent results in RMT are related to this topic in one way or another. For Wigner matrices, as a consequence of the rigidity of the eigenvalues, a large deviation result for $\mathcal{K}_N$ has been established in  \cite{EYY}. It states that $\mathcal{K}_N$ is bounded by $(\log N)^{O(\log\log N)}/N$ with high probability. We also refer to \cite{GT, EKYY, EKYY1, AEK} for some related developments on $\mathcal{K}_N$. More precise bound of $\mathcal{K}_N$ is available for unitarily  invariant ensembles. For instance, in the recent work \cite{CFLW}, it is proved that the bound $\log N/N$ holds for $\mathcal{K}_N$ with high probability, for a class of unitarily invariant ensembles. However, to identify the limiting distribution of $\mathcal{K}_N$ precisely is still far beyond the current studies. It is worth mentioning that $\mathcal{K}_N$ can also be written as
\begin{align*}
\mathcal{K}_N=\max_t \Big|\frac{1}{N\pi}\Im \log\det\big(H-(t+\ii0^+)I_N\big)-F(t)\Big|, 
\end{align*}
with the convention $\Im \log (x+\ii \infty)=\frac{\pi}{2}$. In contrast to the maximum of the random field for the imaginary part of log characteristic polynomial stated above, more research has been devoted to  understanding  the maximum of the random field for the real part of log characteristic polynomial of random matrices \cite{ABB, CMN, FK, PZ, FHK, FS, LP, CZ, La}. But none of these works are on generally distributed Wigner matrices, and also no rigorous result on the limiting distribution (fluctuation) is obtained, although conjectures have been made in  \cite{FHK,FK,FS} for CUE and GUE.  The related study for $\mathcal{A}_N$ is even less. 
In \cite{Rains1997} (see (1.1) therein), Rains discussed a quadratic goodness-of-fit statistic for CUE, which is constructed in the same spirit as $\mathcal{A}_N$. But only the expectation of Rains' statistic was derived in \cite{Rains1997}. In  Appendix \ref{App:A}, we will show that CUE is actually a toy model for the CvM statistic. We can derive the limiting distribution of a CvM statistic explicitly in this case. 

From the application perspective, unlike
the i.i.d. case where the  distribution of a random variable is the central object for
the hypothesis testing, we can use the goodness-of-fit statistics of the spectral distribution to
test the goodness-of-fit for  general features of the matrix ensembles rather than the
limiting spectral distribution itself. For instance, for the sample covariance matrix, one often
uses spectral statistics to test the structure of the population covariance matrix; for a random
graph, one can use spectral statistics of the adjacency matrix to test the graph parameters. Although the limiting distribution of the CvM statistic for Wigner type or covariance type matrices is not available so far, it has already been used in some statistics works, such as \cite{PG, WP}, where the applications are based on  numerical simulation of the CvM statistic.

For Wigner matrices, it is known that  $F_N(t)$ behaves asymptotically as a log-correlated Gaussian field  in the finite-dimensional sense, on both macroscopic and mesoscopic scale;  see \cite{Johansson98, BY,  LS18, Gu05, OR10, BM18} for instance.  However, unlike the i.i.d. case,  this log-correlated Gaussian field asymptotic for random matrix  is not precise enough to tell the fluctuation of $\mathcal{A}_N$ and $\mathcal{K}_N$. It can be viewed as a common difficulty for $\mathcal{A}_N$ and $\mathcal{K}_N$, whose distributions both rely on  a rather delicate understanding of the limiting behaviour of the  field $F_N(t), t\in \mathbb{R}$. 
 On the other hand, an elementary calculation leads to the alternative representation
\begin{align*}
\mathcal{A}_N= \frac{1}{N}\sum_{i=1}^N \Big(F(\lambda_i)-\frac{N-i+\frac12}{N}\Big)^2+\frac{1}{12 N^2}\,,
\end{align*}
which together with the spectral rigidity in \cite{EYY} allows us to write  
\begin{align}
\mathcal{A}_N=\frac{1}{N}\sum_{i=1}^N \rho_{sc}(\mu_i)^2(\lambda_i-\mu_i)^2+\frac{1}{12 N^2}+O(N^{-3+\varepsilon}) \label{19092801}
\end{align}
with high probability, where $\rho_{sc}$ represents the density function of $F$ and $\mu_i$'s represent the quantiles of $F$, i.e., $F(\mu_i)=(N-i+\frac12)/N$. Hence, one can also regard $\mathcal{A}_N$ as certain quadratic eigenvalue statistic.

In this paper, we take a first step to understand the limiting distribution of the CvM type statistics for Wigner matrices. Instead of $\mathcal{A}_N$, we turn to study a more accessible mesoscopic approximation of $\mathcal{A}_N$, see  Definition \ref{defn.19083001}. Our aim is to derive the limiting distribution for this approximated $\mathcal{A}_N$. Numerical study strongly suggests that the fluctuation of this approximation shall be the same as the fluctuation of the original $\mathcal{A}_N$. The mesoscopic approximation can be regarded as a regularization of $F_N(t)$ on mesoscopic scale, which resembles the classical idea in RMT of using the imaginary part of Stieltjes transform on mesoscopic scale to regularize the eigenvalue density. Similar ideas of regularizing the entire  field $F_N(t)$ can also be found in \cite{LOS18, FKS}, for instance.

To study the mesoscopic approximation of $\cal A_N$, we first use the Fourier-Chebyshev expansion to factorize it into linear statistics of Chebychev polynomials. The main step of the proof lies at analyzing the covariance of squares of linear statistics of Chebychev polynomials with $N$-dependent degrees (see Proposition \ref{pro.cov} below). To this end, we explicitly compute the four-point correlation functions for mesoscopic linear statistics of Green functions, by the cumulant expansion formula. We present a simple argument that analyzes the four-point correlation functions both in the bulk and near the edge of the spectrum; the result is precise in the sense that it reveals the cancellations coming from the orthogonality of Chebychev polynomials, for all degrees $k\ll N^{1/3}$.  We remark here that  the fact that the Chebyshev polynomials form an orthogonal basis for the covariance structure of the linear eigenvalue statistics for random Hermitian matrix was first noticed in \cite{Johansson98}, for general $\beta$ ensembles. It was recently rigorously shown in \cite{FKS} that the field of log characteristic polynomial of GUE converges to log correlated Gaussian fields, on macroscopic and mesoscopic scales. Especially, on macroscopic scale,  the limiting log correlated Gaussian field  is given by a random Fourier-Chebyshev series, and the convergence is  understood in the sense of distributions in a suitable Sobolev space. Fourier-Chebyshev expansion of the test function plays a significant role in \cite{FKS} as well.

 Finally, we remark here that except for the classical distances $\mathcal{A}_N$ and $\mathcal{K}_N$ often used in applied probability and statistics, some other  distances  between  $F_N$ and $F$  such as Wasserstein distance has also been considered in the random matrix literature; see \cite{MM13, MM} for instance. Especially, very precise large deviation results are obtained in these works for various random matrix models such as Wigner matrices, Wishart matrices, Haar-distributed matrices from the compact classical groups. We believe that further study on  the explicit limiting distributions for all these statistical distances is appealing from both theoretical and applied point of view.

\subsection{Mesoscopic approximation of CvM statistics and main result}
In this paper, we will turn to study a mesoscopic approximation of the CvM statistic for Wigner matrix. 
The construction of our mesoscopic statistic in Definition \ref{defn.19083001} is inspired by a CvM statistic for the Circular Unitary Ensemble (CUE), whose limiting distribution is established in Appendix  \ref{App:A} in details. Here we provide a brief outline of the derivation for this limiting law for CUE and explains how it inspires the construction in Definition \ref{defn.19083001} for Wigner matrix. 

Let $U$ be a $N$-dimensional CUE, i.e., a Haar distributed random unitary matrix on unitary group $\mathcal{U}(N)$,  and denote its eigenvalues by $\mathrm{e}^{\mathrm{i}\theta_i}, 1\leq i\leq N$. Here $\theta_i$'s are unordered eigenphases. With certain abuse of notation, we still denote by  $F_N(x)=\sharp\{i:  0\leq \theta_i\leq x\}/N, x\in [0,2\pi]$ the empirical distribution of eigenvalues of CUE. The CvM statistic for CUE is defined as 
\begin{align*}
\mathcal{A}_N^{\text{CUE}}:=\int_{0}^{2\pi}\int_{0}^{2\pi} \Big(\big(F_N(y)-F_N(x)\big)-\frac{y-x}{2\pi }\Big)^2 {\rm d} x{\rm d} y
\end{align*}
instead; we refer to Appendix \ref{App:A} for an explanation on the double integral construction. The basic idea of the derivation for the limiting distribution of $\mathcal{A}_N^{\text{CUE}}$ is from \cite{DE01}. 
Using the Fourier transform of the indicator function, one can 
write  for any $x\leq y$, 
\begin{align*}
F_N(y)-F_N(x) - \frac{y-x}{2\pi }\mapsto \frac{1}{2N\pi \mathrm{i}}\sum_{j=1}^\infty \frac{\mathrm{e}^{-\mathrm{i} jx}-\mathrm{e}^{-\mathrm{i} jy}}{j} \text{Tr}U^j
+\frac{1}{2N\pi \mathrm{i}}\sum_{j=1}^\infty \frac{\mathrm{e}^{\mathrm{i} jy}-\mathrm{e}^{\mathrm{i} jx}}{j} \text{Tr}\overline{U}^j.
\end{align*}
Then simple orthonormality leads to
\begin{align}
\mathcal{A}_N^{\text{CUE}}
=
&\frac{4}{N^2}\sum_{j=1}^\infty \frac{1}{j^2}\text{Tr}U^j \text{Tr}\overline{U}^j. \label{19090401}
\end{align}
It is known from \cite{DE01} that the collection $\{\frac{1}{\sqrt{j}} \text{Tr} U^j\}_{j=1}^{N^{1-\varepsilon}}$ behaves like independent complex Gaussian variables, in the sense of moments.  In addition, the expectations and covariances  of   $\text{Tr}U^j \text{Tr}\overline{U}^j$'s can be computed for any $j$ even if $j$ is larger than $N$, say. Such a calculation is elementary by using the explicit kernel for the determinantal point process of the eigenvalues of CUE (c.f. (\ref{20070710})). The result shows that the fluctuations of the large $j$ terms in the sum (\ref{19090401}) are negligible. Hence, the fluctuation of $\mathcal{A}_N^{\text{CUE}}$ is actually governed by the asymptotic Gaussianity of $\{\frac{1}{\sqrt{j}} \text{Tr} U^j\}_{j=1}^{N^{1-\varepsilon}}$ for relatively small $j$'s.  After appropriate  normalization, one can get the limiting law for $\mathcal{A}_N^{\text{CUE}}$ in   (\ref{20070301}).

The above strategy is based on the fact that the basis $\{\text{Tr} U^j\}_{j=1}^\infty$ decomposes the randomness of the linear eigenvalue statistics of CUE, and the information on the joint distribution of $\{\text{Tr} U^j\}_{j=1}^\infty$ can be precisely extracted from  Theorem 2.1 of \cite{DE01} or the explicit jpdf of the eigenvalues of CUE (c.f. (\ref{jpdf of CUE})). It is well known that for Winger matrices, the set of statistics $\{\text{Tr} T_j(H)\}_{j=1}^{\infty}$  serves as an analogue of  $\{\text{Tr} U^j\}_{j=1}^\infty$ for CUE; see \cite{Johansson98} for instance. Here $T_j$'s defined in (\ref{20070310}) are the  Chebyshev polynomials of the first kind. Hence, using the basis $\{\text{Tr} T_j(H)\}_{j=1}^{\infty}$, one can get a similar expansion as (\ref{19090401}) for the CvM statistic of Wigner matrices. Unfortunately, the limiting distribution of $\text{Tr} T_j(H)$'s is only known for fixed $j$; see \cite{Johansson98} and \cite{BY}, and the limiting behaviour of $\text{Tr} T_j(H)$'s with large $j$ is very hard to get due to the lack of  explicit jpdf of eigenvalues for general Wigner matrices. Although our analysis allows us to extend the limiting behaviour of $\text{Tr} T_j(H)$'s from fixed $j$ to some moderately growing $j$, say $j\lesssim N^{\frac13-\varepsilon}$, the analysis does not allows us to establish a limiting result for arbitrarily large $j$. Hence, our mesoscopic statistic defined in  Definition \ref{defn.19083001} is constructed to suppress the contribution of the large $j$ terms of $\text{Tr} T_j(H)$ in the expansion of the CvM statistics. We use the classical idea in Fourier theory to use the convolution with a Poisson kernel to suppress the large $j$ terms. The construction is detailed in the sequel. 

Our statistic is constructed via a Poisson convolution. Specifically, we set the Poisson type kernel on $(-1,1)$ 
\begin{align*}
P^{\pm}_\omega(x,y)=\frac{1-r_\omega^2}{1-2r_\omega(xy\pm\sqrt{(1-x^2)(1-y^2)})+r_\omega^2}, \qquad   r_\omega=1-\omega. 
\end{align*} 
 Then we define the transformation for any real function $f(y)$ which is square integrable on $(-1,1)$ w.r.t. the weight function $1/\sqrt{1-y^2}$, 
\begin{align}
f_\omega(x)= \frac{1}{2\pi}\int_{-1}^1 \big(P^+_\omega(x,y)+P^-_\omega(x,y)\big) \frac{f(y)}{\sqrt{1-y^2}} {\rm d}y.  \label{19082801}
\end{align}
Observe that $f_{\omega}(\cos\theta)$ is the Poisson transform of $f(\cos\theta)$ on $(-\pi,\pi)$. 
Hence,  $f_\omega(x)$ is an approximation of $f(x)$ on scale $\omega$. Especially, for any  $x\in (-1,1)$ which is a continuity point of $f$, we have $f_\omega(x)\to f(x)$ if $\omega\downarrow 0$.  In addition, observe that the domain of the integral in the definition (\ref{190529100}) is $t\in (-1,1)$. For any $t\in (-1,1)$. We can also write 
\begin{align*}
F_N(t)=\frac{1}{N}\sum_{i=1}^N\mathbf{1}(\bar{\lambda}_i\leq t),  
\end{align*} 
where $\bar{x}=-1\vee x\wedge 1$ for any $x\in \mathbb{R}$.  In order to raise our mesoscopic approximation of CvM statistic, we approximate the indicator function $\mathbf{1}(\cdot\leq t)$ by a smooth approximation in the sense of (\ref{19082801}). More specifically, we denote by 
\begin{align}
\chi_\omega^t(x)=\frac{1}{2\pi} \int_{-1}^t \big(P^+_\omega(x,y)+P^-_\omega(x,y)\big) \frac{1}{\sqrt{1-y^2}} {\rm d}y, \quad x\in (-1,1),\quad  t\in (-1,1). \label{190830105}
\end{align}
Then we define the approximation of $F_N(t)$ and $F(t)$ on scale $\omega$ as the following 
\begin{align}
F_{N,\omega}(t):=\frac{1}{N}\sum_{i=1}^N \chi_\omega^t(\bar{\lambda}_i), \qquad F_\omega(t):= \int_{-1}^1 \chi_\omega^t(x){\rm d} F(x).  \label{190830110}
\end{align}
With these approximations, we are ready to define our mesoscopic approximation of $\mathcal{A}_N$.

\begin{defn}[{\it MCvM} statistic] \label{defn.19083001} We call the statistic 
\begin{align*}
\mathcal{A}_{N,\omega}:= \int \big(F_{N,\omega}(t)-F_\omega(t)\big)^2 {\rm d} F(t)
\end{align*}
the mesoscopic approximation of the {\rm{CvM}} statistic  on scale $\omega$, which will be abbreviated as {\it MCvM} statistic (on scale $\omega$). 
\end{defn}

In this paper, we will investigate  the limiting distribution of $\mathcal{A}_{N,\omega}$ when $\omega\equiv \omega_N=N^{-\alpha}$ with some constant $\alpha>0$. Let $(Z_i)$ be a sequence of independent standard real Gaussian random variables. Our main theorem is as following. 
\begin{thm} \label{thm. main result} Let $H$ be a Wigner matrix, and $\cal A_{N,\omega}$ be as in Definition \ref{defn.19083001}. Let $\omega=N^{-\alpha}$. For any fixed $\alpha\in (0, \frac13)$, we have 
\begin{multline} \label{19110601}
N^2\mathcal{A}_{N,\omega}-\frac{\alpha \log N}{\beta\pi^2}-b_{\beta}\overset{d}{\longrightarrow}  \frac{1}{\beta\pi^2}\sum_{k=1}^\infty\bigg( \frac{1}{ k}(Z_k^2-1)-\frac{1}{\sqrt{k(k+2)}}Z_kZ_{k+2}\bigg) \\
+\frac{2-\beta}{\pi^2}\sum_{k=1}^\infty\Big(\frac{k+2}{4k^{3/2}(k+1)}Z_{2k}-\frac{1}{4k\sqrt{k+1}}Z_{2k+2}\Big)+a_\beta\,, 
\end{multline}
where  
\begin{align*}
&a_\beta\deq 
\frac{1}{8\pi^2}\big((4\sigma^2-16c_4-6+\beta)(\beta^{-1}+8c_4)^{1/2}+3(\beta-2)\big)Z_2+\frac{2\sqrt{2}}{\sqrt{\beta}\pi^2}\Big(c_4 -\frac{\sigma^2+\beta-3}{16}\Big)Z_4\\
&-\frac{2}{\sqrt{3\beta}\pi^2}c_4 Z_6+\frac{1}{ \pi^2}\Big(\frac{3(\sigma^2+\beta-3)}{4}+\frac{1}{2\beta}\Big)(Z_1^2-1)-\frac{1}{\sqrt{3}\pi^2}\Big(\frac{\sigma}{\sqrt{2\beta}}-\frac{1}{\beta}\Big)Z_1Z_3\\
&+\frac{4}{\pi^2}c_4(Z_2^2-1)-\frac{1}{2\sqrt{2}\beta\pi^2}\big((1+8c_4\beta)^{1/2}-1\big)Z_2Z_4\,,
\end{align*}
and
\[
b_\beta\deq-\frac{\log 2-1/2}{\beta\pi^2}+(2-\beta)\Big(\frac{1}{48}-\frac{1}{8\pi^2}\Big)+\frac{1}{16\pi^2}(\sigma^2+\beta-3)(2\sigma^2-\beta+12)+\frac{19-2\beta-3\sigma^2}{3\pi^2}c_4
+\frac{8}{\pi^2}c_4^2\,.
\]
\end{thm}
\begin{rem}
	Note that in the Gaussian case where $\sigma^2=3-\beta$ and $c_4=0$, Theorem \ref{thm. main result} simplifies to
	\begin{align*}
	&N^2\mathcal{A}_{N,\omega}-\frac{\alpha \log N}{\beta\pi^2}+\frac{\log 2-1/2}{\beta\pi^2}-(2-\beta)\Big(\frac{1}{48}-\frac{1}{8\pi^2}\Big)\nonumber\\
	&\quad\overset{d}{\longrightarrow}  \frac{1}{\beta\pi^2}\sum_{k=1}^\infty\bigg( \frac{1}{ k}(Z_k^2-1)-\frac{1}{\sqrt{k(k+2)}}Z_kZ_{k+2}\bigg)\nonumber\\
	&\quad\quad+\frac{2-\beta}{\pi^2}\sum_{k=1}^\infty\Big(\frac{k+2}{4k^{3/2}(k+1)}Z_{2k}-\frac{1}{4k\sqrt{k+1}}Z_{2k+2}\Big)+\frac{1}{2\beta \pi^2}(Z_1^2-1)\,. 
	\end{align*}
\end{rem}

\begin{rem} We also remark here that our proof of the main theorem will rely on the local semicircle law on scale $\omega$. We believe that a finer argument shall allow one to push the scale $\omega$ to $\frac{1}{N}$, i.e., $\alpha=1$, which is the optimal scale of local law. However, within the framework of the current proof strategy, getting the limiting distribution of the original $\mathcal{A}_N$ requires one to go even below the scale $\frac{1}{N}$.  
\end{rem}

  Finally, we remark here that although $\frac{\alpha \log N}{\beta\pi^2}+b_{\beta}$ shall be far away from the expectation of the original CvM statistic $\mathcal{A}_N$,  the fluctuation in the RHS  of (\ref{19110601}) is believed to be the same as the fluctuation of $\mathcal{A}_N$. We refer to Figures \ref{fig2} and \ref{fig3} for some simulation results for the random matrices $H=(W+W')/{2\sqrt{2N}}$. Here for Figure \ref{fig2} we take $W=(w_{ij})$ to be a $N\times N$ random matrix with i.i.d. standard real Gaussian elements; for Figure \ref{fig3}, we take $W=(w_{ij})$ to be a $N\times N$ random matrix with i.i.d. standard  Rademacher elements, i.e.,  $\mathbb{P}(w_{ij}=1)=\mathbb{P}(w_{ij}=-1)=\frac12$.  For the simulation purpose, we truncate the series  on the RHS of (\ref{19110601})  at $k=300$ and its density function is given by the blue curve. Green curve is a smooth approximation of the histogram of the original CvM statistic $\mathcal{A}_N$ for $H$ with dimension $N=400$. Both curves are plotted based on 
 6000 repetitions of simulation study.  The red curve is a shift of the Green one to mean $0$.   We notice that the centered histogram of the original $\mathcal{A}_N$ (red curve) matches perfectly the plot of the density of the random variables on the RHS of (\ref{19110601})  (blue curve) in both two figures.

 \begin{figure}[!ht]
           \begin{floatrow}
             \ffigbox{\includegraphics[width=7cm, height=5cm]{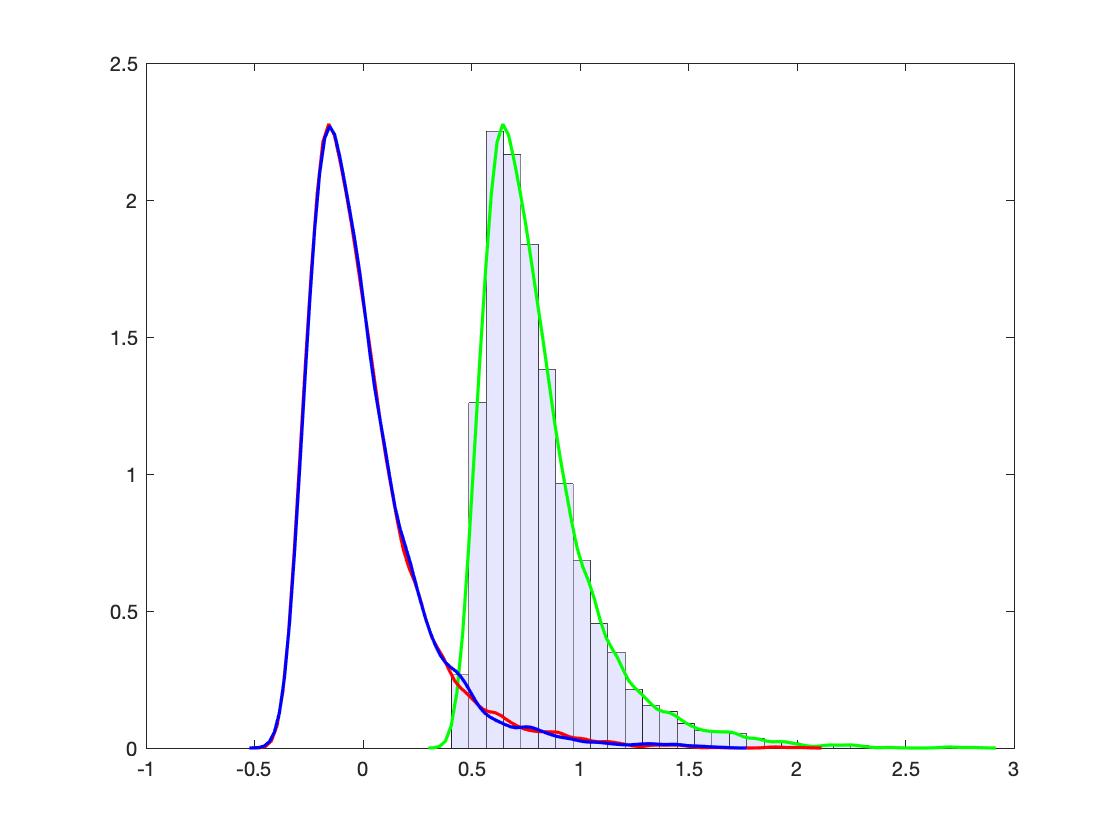}}{\caption{{\footnotesize{Gaussian case} }}\label{fig2}}
             \ffigbox{\includegraphics[width=7cm, height=5cm]{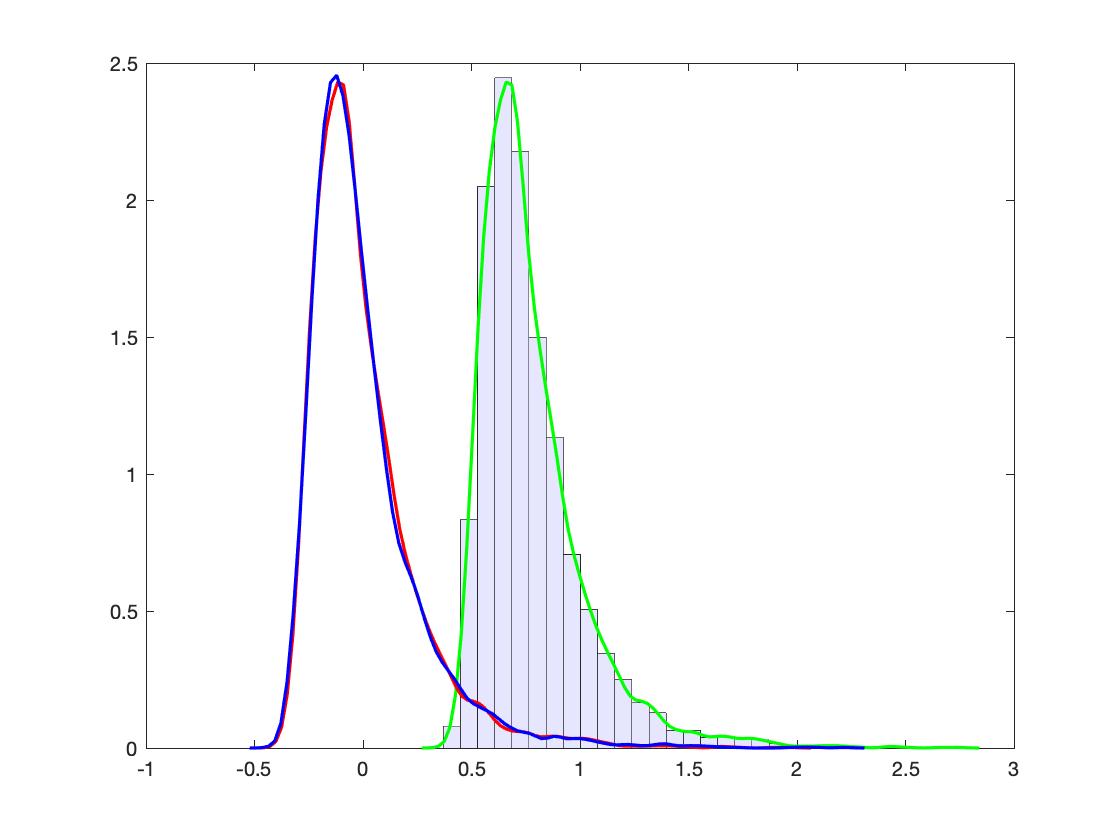}}{\caption{{\footnotesize{Rademacher case}}}\label{fig3}}
           \end{floatrow}
        \end{figure}

\subsection{Organization}   The paper is organized as follows. In Section  \ref{s. preliminary}, we state some preliminaries. Section \ref{sec 3} will be devoted to the proof of the main result, Theorem \ref{thm. main result}, based on Propositions \ref{pro.cov} and \ref{pro.exp}, whose proofs will be stated in Section \ref{sec 4}.  In Section \ref{sec 5}, we prove our main technical result, Proposition \ref{pro.19052401}, which is used in the proofs in Section \ref{sec 4}.  In Section \ref{section 6} we provide some further discussion on the CvM statistics for generally distributed random matrices and in Appendix \ref{App:A} we derive the limiting distribution of a CvM type statistic for the toy model CUE.  

\subsection{Conventions}
 Throughout this paper, we regard $N$ as our fundamental large parameter. Any quantities that are not explicit constant or fixed may depend on N; we almost always omit the argument $N$ from our notation. We use $\|A\|$ to denote the operator norm of a matrix $A$ and use $\|\mathbf{u}\|_2$ to denote the $L^2$-norm of a vector $\mathbf{u}$. We use $\tau$ to denote some generic (small) positive constant, whose value may change from one expression to the next. Similarly, we use $C$ to denote some generic (large) positive constant. For $A \in \bb C$, $B>0$ and parameter $a$, we use $A = O_a(B)$ to denote $|A|\leq 
C_aB$ with some positive constant $C_a$ which may depend on $a$ and $A\asymp B$ to denote $C^{-1}B \leq |A| \leq CB$. When we write $A \ll B$ and $A \gg B$, we mean $|A| \leq CN^{-\tau}B$ and $|A| \geq C^{-1}N^{\tau}B$ for some constants $C,\tau>0$ respectively.

\section{Preliminaries} \label{s. preliminary}

Throughout  the paper, for an $N\times N$ matrix $A$, we write $\ul{A}\deq \frac{1}{N} \tr A$, and we abbreviate $A_{ij}^m\deq (A_{ij})^m$. We emphasize here that $A_{ij}^m$ is different from $(A^m)_{ij}$ in general, where the latter apparently means the $(i,j)$ entry of $A^m$.  For $\b u,\b v \in \bb C^{N}$, we abbreviate
\begin{align}
A_{\b u\b v}\deq \langle \b u , A\b v \rangle  \quad A_{\b u i}\deq \langle \b u , A\b e_i \ra\quad  \mbox{and} \quad  A_{i \b u}\deq \langle \b e_i , A\b u \ra\,, \label{notation for quadratic form}
\end{align}
where $\b e_i$ is the standard $i$-th basis vector of $\bb R^{N}$. We denote $\langle X \rangle \deq X-\bb E X$ for any random variable $X$ with finite expectation.

\subsection{Green function and the local semicircle law}
For $z \in \bb C\setminus \mathbb{R}$, we denote the Green function of $H$ and the Stieltjes transform of its empirical eigenvalue distribution by 
\begin{align*}
G(z)\deq(H-z)^{-1},  \quad \ul{G}(z) =\frac{1}{N}\tr G(z)=\frac{1}{N}\sum_{i=1}^N \frac{1}{\lambda_i-z}\,.
\end{align*}
Correspondingly, we denote the   Stieltjes transform  of the semicircle law by 
\begin{align}
 m(z)= \int \frac{1}{x-z} \rho_{sc}(x) {\rm d} x=2(-z+\sqrt{z^2-1})\,. \label{19101520}
\end{align}
For any positive integer $n$, we use the shorthand notation $\llbracket 1, n\rrbracket\deq \{1,2,...,n\}$. We also adopt the notion of {\it stochastic domination}  introduced in \cite{EKY13}. It provides a convenient way of making precise statements of the form ``$\mathsf{X}^{(N)}$ is bounded by $\mathsf{Y}^{(N)}$ up to  small powers of $N$ with high probability".
\begin{defn} [Stochastic domination]  \label{def2.1}
	Let
	\begin{equation*}
	\mathsf{X}=(\mathsf{X}^{(N)}(u):  N \in \mathbb{N}, \ u \in \mathsf{U}^{(N)}), \   \mathsf{Y}=(\mathsf{Y}^{(N)}(u):  N \in \mathbb{N}, \ u \in \mathsf{U}^{(N)}),
	\end{equation*}
	be two families of random variables, where $\mathsf{Y}$ is nonnegative, and $\mathsf{U}^{(N)}$ is a possibly $N$-dependent parameter set. We say that $\mathsf{X}$ is stochastically dominated by $\mathsf{Y},$ uniformly in $u,$ if for all small $\epsilon>0$ and large $ D>0$, we have 
	\begin{equation*}
	\sup_{u \in \mathsf{U}^{(N)}} \mathbb{P} \Big( \big|\mathsf{X}^{(N)}(u)\big|>N^{\epsilon}\mathsf{Y}^{(N)}(u) \Big) \leq N^{- D},
	\end{equation*}   
	for large enough $N \geq  N_0(\epsilon, D).$  If $\mathsf{X}$ is stochastically dominated by $\mathsf{Y}$, uniformly in $u$, we use the
	notation $\mathsf{X} \prec \mathsf{Y}$ , or equivalently $\mathsf{X} = O_{\prec}(\mathsf{Y})$. Note that in the special case when $\mathsf{X}$ and $\mathsf{Y}$ are deterministic, $\mathsf{X} \prec\mathsf{Y}$
	means that for any given $\varepsilon>0$,  $|\mathsf{X}^{(N)}(u)|\leq N^{\epsilon}\mathsf{Y}^{(N)}(u)$ uniformly in $u$, for all sufficiently large $N\geq N_0(\epsilon)$.
	
	Throughout this paper, the stochastic domination will always be uniform in
	all parameters (mostly are matrix indices and the spectral parameter $z$) that are not explicitly fixed.
\end{defn}

We have the following elementary result about stochastic domination.

\begin{lem} \label{prop_prec} Let
	\begin{equation*}
	\mathsf{X}_i=(\mathsf{X}^{(N)}_i(u):  N \in \mathbb{N}, \ u \in \mathsf{U}^{(N)}), \   \mathsf{Y}_i=(\mathsf{Y}_i^{(N)}(u):  N \in \mathbb{N}, \ u \in \mathsf{U}^{(N)}),\quad i=1,2
	\end{equation*}
	be families of  random variables, where $\mathsf{Y}_i, i=1,2,$ are nonnegative, and $\mathsf{U}^{(N)}$ is a possibly $N$-dependent parameter set.	Let 
	\begin{align*}
	\Phi=(\Phi^{(N)}(u): N \in \mathbb{N}, \ u \in \mathsf{U}^{(N)})
	\end{align*}
	be a family of deterministic nonnegative quantities. We have the following results:
	
(i)	If $\mathsf{X}_1 \prec \mathsf{Y}_1$ and $\mathsf{X}_2 \prec \mathsf{Y}_2$ then $\mathsf{X}_1+\mathsf{X}_2 \prec \mathsf{Y}_1+\mathsf{Y}_2$ and  $\mathsf{X}_1 \mathsf{X}_2 \prec \mathsf{Y}_1 \mathsf{Y}_2$.

 (ii) Suppose $\mathsf{X}_1 \prec \Phi$, and there exists a constant $C>0$ such that  $|\mathsf{X}_1^{(N)}(u)| \leq N^{C}\Phi^{(N)}(u)$ a.s. uniformly in $u$ for all sufficiently large $N$. Then $\E \mathsf{X}_1 \prec \Phi$. 
\end{lem}

\begin{proof}
	Part (i) is obvious from Definition \ref{def2.1}. For any fixed $\varepsilon>0$, we have
	\[
	|\bb E \mathsf{X}_1| \leq \bb E |\mathsf{X}_1{\bf 1}(|\mathsf{X}_1|\leq  N^{\varepsilon}\Phi)|+\bb E |\mathsf{X}_1{\bf 1}(|\mathsf{X}_1|\geq  N^{\varepsilon}\Phi)|\leq N^{\varepsilon}\Phi+ N^{C}\Phi \bb P(|\mathsf{X}_1|\geq N^{\varepsilon}\Phi)=O(N^{\varepsilon}\Phi)
	\] 
 for sufficiently large $N\geq N_0(\epsilon)$. This proves part (ii). 
\end{proof}

Fix $\tau>0$, let us define the spectral domains
\begin{equation*}
\mathcal{S}\deq\big\{ E+\ii \eta: |E|\leq 10, 0 < \eta \leq 10\big\}\quad
\end{equation*}
and
\begin{equation*} \quad \mathcal{S}^{o}\equiv \mathcal{S}^{o}(\tau)\deq\big\{ E+\ii \eta\in \cal S: |E|\geq 1+N^{-2/3+\tau}\big\}\,.
\end{equation*}
We also define the distance to spectral edge by
\[
\kappa \equiv \kappa_E\deq |E^2-1|\,.
\]
We have the following isotropic local semicircle law for Wigner matrices from \cite[Theorems 2.2, 2.3]{KY13} and \cite[Theorem 10.3]{BK16}. 

\begin{thm}[Local semicircle law] \label{thm. local semicircle law} Let $\b u, \b v \in \bb C^{N}$ be deterministic with $\|\b u\|_2=\|\b v\|_2=1$. For Green function, we have the following estimates 
	\begin{equation} \label{2.2}
	|G_{\b u\b v}(z)-m(z)\langle \b u, \b v\rangle|\prec \sqrt{\frac{\Im m(z)}{N\eta}}+\frac{1}{N\eta} \quad \mbox{and} \quad 
	|\ul{G}(z)-m(z)|\prec  \frac{1}{N\eta}
	\end{equation}
	uniformly for $z=E+\ii \eta \in \mathcal{S}$. Moreover, outside the bulk of the spectrum, we have the stronger estimates
	\begin{equation*}
	|G_{\b u\b v}(z)-m(z)\langle \b u, \b v\rangle|\prec \frac{1}{\sqrt{N}(\kappa+\eta)^{1/4}} \quad \mbox{and} \quad
	|\ul{G}(z)-m(z)|\prec  \frac{1}{N(\eta+\kappa)} \label{19083101}
	\end{equation*}
	uniformly for $z=E+\ii \eta \in \mathcal{S}^{o}$. 
\end{thm}

\begin{rem}
(i) From Theorem \ref{thm. local semicircle law}, one can easily deduce (see e.g.\ \cite[Theorems 2.8, 2.9]{BK16})  that
\begin{align}
|\lambda_i-\mu_i|\prec N^{-\frac23} \big(i\wedge (N-i+1)\big)^{-\frac13} \label{19052220}
\end{align}
uniformly in $i\in \llbracket 1, N \rrbracket$, and
\begin{align}
\Big|\big|\big\{i:\lambda_i\in \mathcal{I}\big\}\big|-N\int_{\mathcal{I}} \rho_{sc}(x) {\rm d} x\Big|\prec 1  \label{19052221}
\end{align}
uniformly for any interval $\mathcal{I}\subset \mathbb{R}$.

(ii) Combining with Lemma \ref{prop_prec} and the deterministic bound $\ul{G}(z)\leq \frac{1}{\eta}$, Theorem \ref{thm. local semicircle law} can be used to estimate the moments of the Green function. 
For example, \eqref{2.2} and Lemma \ref{prop_prec} imply
\begin{equation} \label{2.5}
\bb E 	|\ul{G}(z)-m(z)|^n\prec  \Big(\frac{1}{N\eta}\Big)^n
\end{equation}
for any fixed $n \in \bb N$. Note that 
$$
|\ul{G}-\bb E \ul{G}|\leq |\ul{G}-m(z)|+\bb E |\ul{G}-m(z)| \prec \frac{1}{N\eta}\,,
$$
 thus the shift $m(z)$ in \eqref{2.5} can also be replaced by the expectation $\bb E \ul{G}$. 
\end{rem}

If $h$ is a real-valued random variable with finite moments
of all order, we denote by $\cal C_n(h)$ the $n$th cumulant of $h$, i.e.
\[
\cal C_n(h)\deq (-\ii)^n \cdot\Big(\frac{{\rm d}^n}{{\rm d} \lambda^n} \log \bb E \e{\ii  \lambda h}\Big) \Big{|}_{\lambda=0}\,.
\]
Below we state the cumulant expansion formula, whose proof is given in e.g. \cite[Appendix A]{HKR}. 
\begin{lem}[Cumulant expansion] \label{lem:cumulant_expansion}
	Let $f:\R\to\C$ be a smooth function, and denote by $f^{(n)}$ its $n$th derivative. Then, for every fixed $\ell \in\N$, we have 
	\begin{equation}\label{eq:cumulant_expansion}
	\mathbb{E}\big[h\cdot f(h)\big]=\sum_{n=0}^{\ell}\frac{1}{n!}\mathcal{C}_{n+1}(h)\mathbb{E}[f^{(n)}(h)]+\cal R_{\ell+1},
	\end{equation}	
	assuming that all expectations in \eqref{eq:cumulant_expansion} exist, where $\cal R_{\ell+1}$ is a remainder term (depending on $f$ and $h$), such that for any $t>0$,
	\begin{equation*}  
	\cal R_{\ell+1} = O(1) \cdot \bigg(\E\sup_{|x| \le |h|} \big|f^{(\ell+1)}(x)\big|^2 \cdot \E \,\big| h^{2\ell+4} \mathbf{1}(|h|>t) \big| \bigg)^{1/2} +O(1) \cdot \bb E |h|^{\ell+2} \cdot  \sup_{|x| \le t}\big|f^{(\ell+1)}(x)\big|\,.
	\end{equation*}
\end{lem}

The following result gives bounds on the cumulants of the entries of $H$, whose proof follows by the homogeneity of the cumulants.
\begin{lem} \label{Tlemh}
	For every $n \in \bb N^*$ we have
	\begin{equation*}
	\cal C_{n}(H_{ij})=O_{n}(N^{-n/2})
	\end{equation*}
	uniformly for all $i,j$.
\end{lem}

We conclude this subsection with a standard complex analysis result from \cite{Davies}. 

\begin{lem}[Helffer-Sj\"{o}strand formula] \label{HS}
	Let $q \in {C}^{2}(\bb R)$, and let $\tilde{q}$ be the almost analytic extension of $q$ defined by
	\begin{equation} \label{eyaya}
	\tilde{q}(x+\ii y)\deq q(x) +\ii yq^{\prime}(x)\,.
	\end{equation}
	Let $\chi \in C^{\infty}_c(\R)$ be an arbitrary cutoff function satisfying $\chi(0) = 1$, and by a slight abuse of notation write $\chi(z) \equiv \chi (\im z)$.
	Then for any $\lambda \in \bb R$ we have
	\begin{equation*} 
	q(\lambda)=\frac{1}{\pi}\int_{\bb C}\frac{\partial_{\bar{z}}(\tilde{q}(z)\chi(z))}{\lambda-z}\,\dd^2z\,,
	\end{equation*}
	where $\partial_{\bar{z}}\deq \frac{1}{2}(\partial_x+\mathrm{i}\partial_y)$ is the antiholomorphic derivative and $\dd^2 z$ the Lebesgue measure on $\C$.
\end{lem}

\subsection{Chebyshev's Polynomial}
Suppose that a  function $f:[-1,1]\to \mathbb{R}$ is square integrable on $[-1,1]$ w.r.t to the weight function $1/\sqrt{1-x^2}$. 
In the sequel, we will consider the Fourier-Chebyshev expansion of  $f$ which admits
\begin{align}
f(x)=\frac{1}{2}c_0+ \sum_{k=1}^\infty c_kT_k(x)=:{\sum'}_{k} c_k T_k(x), \qquad a.e.\label{190830100}
\end{align}
where we used the notation $\sum'_k$ to denote the sum from $k=0$ to $\infty$ with the first summand ($k=0$) halved and $T_k$'s are the Chebyshev polynomials of the first kind, i.e.,
\begin{align}
T_k(\cos\theta)=\cos(k\theta). \label{20070310}
\end{align}
 Here the coefficients $c_k$'s are defined as
\begin{align}
c_k= \frac{2}{\pi}\int_{-1}^1 f(s) T_k(s) \frac{1}{\sqrt{1-s^2}}{\rm d} s. \label{def of coefficient}
\end{align}
By setting $g(\theta)=f(\cos\theta)$ for $0\leq \theta\leq \pi$ and requiring $g(\theta+2\pi)=g(\theta)$ together with $g(-\theta)=\theta$, 
one gets an even periodic function $g$ of period $2\pi$. Then the Fourier-Chebyshev expansion of $f(x)$ in (\ref{190830100}) is equivalent to the Fourier expansion of $g(\theta)$ by identifying $x$ with $\cos\theta$. Hence, the theory of Fourier series can be applied to the Fourier-Chebyshev expansion. Especially,  the identity in (\ref{190830100}) holds in the almost everywhere sense by Carleson's theorem.  We refer to  Chapter 5 of the monograph \cite{MH} for a more detailed introduction of the Fourier-Chebyshev expansion.  
In the whole $ \mathbb{C}$, one can also write the Chebyshev polynomials of the first kind as
\begin{align}
T_k(z)= \frac{1}{2} \Big( \big(z-\sqrt{z^2-1}\big)^k+\big(z+\sqrt{z^2-1}\big)^k\Big),\qquad z\in \mathbb{C} \label{19082902}
\end{align} 
where the square root $\sqrt{z^2-1}$ is chosen with a branch cut in the segment $[-1,1]$ so that $\sqrt{z^2-1}\sim z$
as $z\to \infty$. The above representation is easy to check by setting $z=\cos \theta=(\mathrm{e}^{\ii\theta}+\mathrm{e}^{-\ii\theta})/2$ when $z\in[-1,1]$, and as a polynomial the extension of the representation to all $z\in \mathbb{C}$ is obvious. Note from (\ref{19101520}) that 
\begin{align*}
T_k(z)=\frac{(-1)^k}{2}\Big(\Big(\frac{m(z)}{2}\Big)^{k}+\Big(\frac{m(z)}{2}\Big)^{-k}\Big)\,.
\end{align*}
Let us denote
\begin{align}
t_k(A):= \frac{1}{N}\text{Tr} T_k(A)- \int T_k(x)\rho_{sc}(x) {\rm d} x \label{19052201}
\end{align}
for a matrix $A\in \mathbb{C}^{N\times N}$.  We shall use the following results from \cite{BY}. 

\begin{thm}[Corollary 6.1 of \cite{BY}] \label{thm.by}For any fixed $k\in \mathbb{N}$, the random vector $(Nt_1(H), \ldots, Nt_k(H))$ converges weakly to Gaussian vector $(g_1, \ldots, g_k)$ with independent components and the means and variances are  given by 
	\begin{align*}
	&\mathbb{E}g_k=\frac{2-\beta}{4}(1+(-1)^{k})+\frac{1}{2} (\sigma^2+\beta-3)\delta_{k 2}+8c_4 \delta_{k 4}, \nonumber\\
	& \mathrm{Var}(g_k)= \frac14 \big((3-\beta) k+(\sigma^2+\beta-3)\delta_{k 1}+32c_4 \delta_{k 2}\big),
	\end{align*}
	where the parameters $\beta$, $\sigma^2$ and $c_4$ are defined in Section \ref{s.background}. 
\end{thm}

Finally, with the eigendecomposition $H=\sum_{i} \lambda_i \b u_i \b u_i^*$, we set 
\[
\bar{H} \deq \sum_{i} \bar{ \lambda}_i  \b u_i \b u_i^* \,
\]
Hereafter  we set the notation 
\begin{align}
\bar{x}:=-1\vee x\wedge 1 \label{20070350}
\end{align}
for any  $x\in \mathbb{R}$. 
We have the following comparison result.

\begin{lem} \label{cor. with bar} 
For $k \ll N^{1/3}$, we have 
	\[
	t_k(H)-t_k(\bar{H}) \prec k^2N^{-5/3}\,.
	\]
\end{lem}
\begin{proof} 
	By definition in (\ref{19052201}), we have 
	\begin{align*}
	N(t_k(H)-t_k(\bar{H}))&= \sum_{i=1}^N \Big(T_k(\lambda_i)-T_k(\bar{\lambda}_i)\Big)=  \sum_{i=1}^N \Big(T_k(\lambda_i)-T_k(\bar{\lambda}_i)\Big)\mathbf{1} (|\lambda_i|>1)\nonumber\\
	&\prec k^2\max_i|\lambda_i-1| \big|\big\{ i: |\lambda_i|>1\big\}\big|\prec k^2N^{-\frac23} \,. 
	\end{align*}
	where in the last two steps we used (\ref{19052220}) and its consequence
	(\ref{19052221}). 
\end{proof}

\section{Proof of Theorem \ref{thm. main result}} \label{sec 3}
For the rest of this paper we set
\begin{align*}
\zeta\deq (1/3-\alpha) \wedge \alpha>0\,.
\end{align*}
Let $f$ be a function satisfying (\ref{190830100}). Recall the definition of $f_{\omega}$ in (\ref{19082801}). We have 
\begin{align}
f_{\omega}(\cos \theta) &=   \frac{1}{2\pi}\int_{0}^\pi \Big(\frac{1-r_\omega^2}{1-2r_\omega\cos(\theta-t)+r_\omega^2}+\frac{1-r_\omega^2}{1-2r_\omega\cos(\theta+t)+r_\omega^2}\Big)  f(\cos t) {\rm d} t\nonumber\\
&=\frac{1}{\pi}\int_{0}^\pi \Big(1+\sum_{k=1}^{\infty}r_\omega^{k} \big(\cos(k (\theta-t))+\cos(k (\theta+t))\big)\Big)f(\cos t) {\rm d} t\nonumber\\
&= \frac{1}{\pi} \int_{0}^\pi \Big(1+2\sum_{k=1}^{\infty}r_\omega^{k} \cos(k \theta)\cos(k t)\Big)f(\cos t) {\rm d} t\nonumber\\
&=\frac{1}{2}c_0+\sum_{k=1}^\infty c_k r_{\omega}^k \cos (k\theta), \label{20070401}
\end{align}
 where in the last step we used (\ref{def of coefficient}), and in the second step we used the following elementary identity for Poisson kernel
\begin{align*}
\frac{1-r^2}{1-2r\cos\phi+r^2}=1+2\sum_{k=1}^\infty r^k \cos (k\phi), \qquad 0\leq r<1. 
\end{align*} 
Also, in the last step of (\ref{20070401}), we interchanged the sum over $k$ with the integral over $t$, which can be justified by the uniform convergence of the series. 
Hence, for any $x\in [-1,1]$, we have 
\begin{align*}
f_\omega(x)=\frac12 c_0+\sum_{k=1}^\infty c_k r_{\omega}^k T_k(x).  
\end{align*}
Recall the notation in (\ref{20070350}). Set
\begin{align*}
f_t(\bar{x}):= \mathbf{1}(\bar{x}\leq t).
\end{align*} 
Observe that $f_t(\bar{x})=f_t(x)$ for any $t\in (-1,1)$. Note that $\chi_{\omega}^t(\bar{x})=(f_t)_{\omega}(\bar{x})$, see (\ref{190830105}) for the definition of $\chi_{\omega}^t(x)$.  It is then easy to compute 
\begin{align}
&\chi_{\omega}^t(\bar{x})= \frac12 d_0^t+\sum_{k=1}^\infty d_k^t r_{\omega}^k T_k(\bar{x}), \nonumber\\
& d_k^t:= \frac{2}{\pi}\int_{-1}^t T_k(s) \frac{1}{\sqrt{1-s^2}}{\rm d} s= -\frac{2}{\pi k} \sin (k\cos^{-1} t).   \label{19052101}
\end{align}

According to the definitions in (\ref{19052101}) and (\ref{190830110}), we can write 
\begin{align*}
F_{N,\omega}(t)= \frac12 d_0^t+\frac{1}{N}\sum_{k=1}^\infty d_k^t r_{\omega}^k \text{Tr} T_k(\bar{H}), \qquad F_\omega(t)= \frac12 d_0^t+\sum_{k=1}^\infty d_k^t r_{\omega}^k \int T_k(x)\rho_{sc}(x) {\rm d} x.
\end{align*}
Therefore, from Definition \ref{defn.19083001}, we have 
\begin{align*}
\mathcal{A}_{N,\omega}= \int \big(F_{N,\omega}(t)- F_\omega(t)\big)^2 {\rm d} F(t)= \sum_{j,k=1}^\infty r_{\omega}^{j+k}t_j(\bar{H})t_k(\bar{H}) \int d_j^t d_k^t {\rm d} F(t),
\end{align*}
where we used the notation in (\ref{19052201}).  

In light of the definition of $d_k^t$ in (\ref{19052101}), we have 
\begin{align*}
\int d_j^t d_k^t {\rm d} F(t)&= \frac{8}{\pi^3 jk} \int_{-1}^1 \sin (j\cos^{-1} t)\sin (k\cos^{-1} t)\sqrt{1-t^2}{\rm d} t\nonumber\\
&=  \frac{8}{\pi^3 jk} \int_{0}^{\pi} \sin (j\theta)\sin (k\theta)(\sin \theta)^2{\rm d} \theta\nonumber\\
&=\frac{2}{\pi^2 jk}\Big(\mathbf{1}(j=k)-\frac12\mathbf{1}(j=k\pm 2)+\frac12 \mathbf{1}(j+k=2)\Big).
\end{align*}
Therefore, we have 
\begin{align}
\mathcal{A}_{N,\omega} 
&= \frac{2}{\pi^2}\sum_{k=1}^\infty \frac{1}{k^2} r_{\omega}^{2k} \big(t_k(\bar{H})\big)^2- \frac{2}{\pi^2}\sum_{k=1}^\infty r_{\omega}^{2k+2}\frac{1}{k(k+2)}t_k(\bar{H})t_{k+2}(\bar{H})+\frac{1}{\pi^2} r_{\omega}^{2}\big(t_1(\bar{H})\big)^2. \label{19052202}
\end{align}
According to the definition in (\ref{19052201}), we have $|t_j(\bar{H})|\leq 2$. Further, recall $r_\omega=1-\omega$ and $\omega=N^{-\alpha}$. We can trivially truncate the sum in (\ref{19052202}) to 
\begin{align} 
\mathcal{A}_{N,\omega}= &\frac{2}{\pi^2}\sum_{k=1}^{n_\omega} \frac{1}{k^2} r_{\omega}^{2k} \big(t_k(\bar{H})\big)^2- \frac{2}{\pi^2}\sum_{k=1}^{n_\omega} r_{\omega}^{2k+2}\frac{1}{k(k+2)}t_k(\bar{H})t_{k+2}(\bar{H})+\frac{1}{\pi^2} r_{\omega}^{2}\big(t_1(\bar{H})\big)^2\black+o(N^{-K}) \label{19052210}
\end{align}
for any large constant $K$ when $N$ is sufficiently large. 
Here 
\begin{align*}
n_\omega:= \lfloor \omega^{-1}(\log N)^2 \rfloor. 
\end{align*}
Now for $k\in \llbracket 1, n_\omega\rrbracket$  and  fixed  $a>0$, we define 
\begin{equation*}
\gamma^{(1)}_{a,k}\deq \{E+\ii \eta \in \bb C:E \in (-1,1), \eta=\pm ak^{-1}\sqrt{1-E^2+k^{-2}}\}\,,
\end{equation*}
and 
\begin{equation*}
\gamma^{(2)}_{a,k}\deq \{E+\ii \eta \in \bb C:|E|\geq 1, \dist (\gamma_{a,k}^{(2)},\{-1,1\})=ak^{-2}\}\,.
\end{equation*}
We then consider  the (counterclockwise) contour
\begin{equation*}
\gamma_{a,k}\deq \gamma^{(1)}_{a,k}\cup \gamma^{(2)}_{a,k}
\end{equation*} 
whose figure is sketched below

\begin{center} 
	\begin{tikzpicture}
	[>=triangle 45, scale=2.5]
	\draw[thin,-Latex] (-1.4,0)--(1.4,0) node[below] {$x$}; 
	\draw[thin,-Latex] (0,-0.6)--(0,0.6) node[left] {$y$}; 
	\coordinate  (A) at (-1,0.1);
	\coordinate  (B) at (0,0.3);
	\coordinate (C) at (1,0.1);
	
	\draw[thin,-Latex] plot [smooth] coordinates { (C) (B) (A) };
	
	\coordinate  (D) at (-1,-0.1);
	\coordinate  (E) at (0,-0.3);
	\coordinate (F) at (1,-0.1);
	
	\draw[thin] plot [smooth] coordinates { (D) (E) (F) };
	
	\draw[thin] plot [smooth] coordinates { (A) (B) (C) };

	\draw (1,-0.1) arc (-90:90:0.1);
	\draw (-1,0.1) arc (90:270:0.1);

	\coordinate [label=below:$-1$] (X) at (-1,0);
	\coordinate [label=below:$1$] (Y) at (1,0);
	\node[point] at (X) {};
	\node[point] at (Y) {};
	\end{tikzpicture}
\end{center}

 We further define $\widehat{\gamma}_{a,k}$, which is obtained from $\gamma_{a,k}$ by deleting the part $|\eta|< N^{-10}$, i.e.,
\begin{equation}  \label{3.4}
\widehat{\gamma}_{a,k} \deq \{E+\ii\eta \in \gamma_{a,k}: |\eta|\geq N^{-10}\}\,.
\end{equation} 
The induced path-integral on $\widehat{\gamma}_{a,k} $ is defined by
\[
\oint_{\widehat{\gamma}_{a,k}} f(z) \dd z\deq \oint_{{\gamma}_{a,k}}  {\bf 1}_{\widehat{\gamma}_{a,k}}(z) f(z) \dd z\,,
\]
where ${\bf 1}_{\widehat{\gamma}_{a,k}}(z)=1$ if $z \in \widehat{\gamma}_{a,k}$ and ${\bf 1}_{\widehat{\gamma}_{a,k}}(z)=0$ otherwise. Accordingly, we set 
\begin{align}
\widehat{t}_{k}(H):=\frac{\ii}{2\pi} \oint_{\widehat{\gamma}_{a,k}} T_k(z) \big(\ul{G}(z)-m(z)\big){\rm d }z=: \frac{\ii}{2\pi} \oint_{\widehat{\gamma}_{a,k}} T_k(z) m_N^\Delta(z){\rm d }z\,. \label{def. t hat}
\end{align}

By (\ref{19052220}) and Cauchy's integral formula, with high probability, we have 
\begin{equation} \label{4.1}
\begin{aligned}
t_{k}(H)=&\,\frac{\ii}{2\pi} \oint_{{\gamma}_{a,k}} T_k(z) \big(\ul{G}(z)-m(z)\big){\rm d }z \\
=&\,\widehat{t}_{k}(H)+\frac{\ii}{2\pi} \oint_{{\gamma}_{a,k}} T_k(z) \big(\ul{G}(z)-m(z)\big)(1-{\bf 1}_{\widehat{\gamma}_{a,k}}(z)){\rm d }z\,.
\end{aligned}
\end{equation}
For $z \in \gamma_{a,k}^{(1)}$ satisfying $1-E^2 \geq |\eta|$, we have
\begin{multline} \label{zuoyi}
z\pm \sqrt{z^2-1}=E+\ii \eta \pm \sqrt{E^2-1} \sqrt{1-\frac{2E\eta}{1-E^2}\ii+\frac{\eta^2}{1-E^2}}\\
=E+\ii \eta \pm \sqrt{E^2-1} \Big(1+O_a\Big(\frac{|\eta|}{1-E^2}\Big)\Big)=E\pm\sqrt{E^2-1}+O_a(k^{-1})\,,
\end{multline}
where in the second step we used $|\eta^2|=O_a(|\eta|)$.  For $z \in \gamma_{a,k}^{(1)}$ satisfying $1-E^2 < |\eta|$ or  $z \in \gamma_{a,k}^{(2)}$, we have $|E^2-1|,|\eta|=O_a(k^{-2})$. Thus
\begin{equation} \label{wubai}
z\pm \sqrt{z^2-1}=E+\ii \eta \pm \sqrt{E^2-1+2E \eta \ii -\eta^2}=1+O_a(k^{-1})\,.
\end{equation}
The estimates  \eqref{zuoyi} and \eqref{wubai}, together with (\ref{19082902}),  imply
\begin{align} \label{eqn: bound T}
\sup_{z\in \gamma_{a,k}} |T_k(z)|=O_a(1)
\end{align}
for all $k\leq n_\omega$.  Furthermore, by Theorem \ref{thm. local semicircle law} and \eqref{4.1} one easily deduces that
\begin{align*}
t_{k}(H) \prec \frac{k}{N}\quad \mbox{and} \quad t_{k}(H)=\widehat{t}_{k}(H)+O_{\prec}(N^{-10}) \,.  
\end{align*} 
Together with Lemma \ref{cor. with bar}, \eqref{19052210}, and $\omega N^{1/3}\geq N^{\zeta}$, we arrive at
\begin{align*}
\mathcal{A}_{N,\omega}= &\frac{2}{\pi^2}\sum_{k=1}^{n_\omega} \frac{1}{k^2} r_{\omega}^{2k} \big(t_k({H})\big)^2- \frac{2}{\pi^2}\sum_{k=1}^{n_\omega} r_{\omega}^{2k+2}\frac{1}{k(k+2)}t_k({H})t_{k+2}({H})+\frac{1}{\pi^2} r_{\omega}^{2}\big(t_1({H})\big)^2\black+O_{\prec}(N^{-2-2\zeta})\nonumber\\
= &\frac{2}{\pi^2}\sum_{k=1}^{n_\omega} \frac{1}{k^2} r_{\omega}^{2k} \big(\widehat{t}_k({H})\big)^2- \frac{2}{\pi^2}\sum_{k=1}^{n_\omega} r_{\omega}^{2k+2}\frac{1}{k(k+2)}\widehat{t}_k({H})\widehat{t}_{k+2}({H})+\frac{1}{\pi^2} r_{\omega}^{2}\big(\widehat{t}_1({H})\big)^2\black+O_{\prec}(N^{-2-2\zeta}).
\end{align*}
 By the trivial bound $|m^{\Delta}(z)|\leq |m(z)|+|\ul{G}(z)|\leq 2|\eta|^{-1}\leq 2N^{10}$, we also have the deterministic bound
\begin{equation*}
|\widehat{t}_k(H)| \leq C N^{10}\,.
\end{equation*}
 For the rest of this paper, we shall only work on the case when $H$ is real and symmetric ($\beta=1$); in the complex Hermitian case ($\beta=2$), one only needs to apply the complex analogue of Lemma \ref{lem:cumulant_expansion} (see e.g.\,\cite[Lemma 7.1]{HK16}) and the proof works in the same way. Therefore, without further explanation, the discussions from the rest of this section till Section \ref{sec 5} will be stated for the real $H$ only.   
 
Now  Theorem \ref{thm. main result} (the real case) follows easily from the following results, whose proofs will be given in Section \ref{sec 4}.

\begin{pro} \label{pro.cov}For any  $k,l\in\llbracket 1, n_\omega\rrbracket$, we have 
	\begin{align}\label{19012420}
	&\mathrm{Cov} \big((\widehat{t}_k({H}))^2, (\widehat{t}_l({H}))^2\big)\nonumber\\
	&=
	4\bigg(\frac{1}{4N}(1+(-1)^k)+\frac{8 c_4 }{N} \delta_{k4}+\frac{\sigma^2-2}{2N}\delta_{k2}\bigg)^2\bigg( \frac{ k}{2N^2}+\frac{8c_4}{N^2}\delta_{k 2}+\frac{\sigma^2-2}{4N^2}\delta_{k1}\bigg)\delta_{kl} \nonumber\\
	&+2\bigg( \frac{ k}{2N^2}+\frac{8c_4}{N^2}\delta_{k 2}+\frac{\sigma^2-2}{4N^2}\delta_{k1}\bigg)^2\delta_{kl}+O_{\prec}(klN^{-4-\zeta})
	\end{align}
	and
	\begin{align*}
	&\mathrm{Cov} \big( \widehat{t}_k({H})\widehat{t}_{k+2}({H}),\widehat{t}_l({H})\widehat{t}_{l+2}({H})\big) \nonumber\\
	&=
	\bigg(\frac{1}{4N}(1+(-1)^{k+2})+\frac{8 c_4 }{N} \delta_{k2}\bigg)^2\bigg( \frac{ k}{2N^2}+\frac{8c_4}{N^2}\delta_{k 2}+\frac{\sigma^2-2}{4N^2}\delta_{k1}\bigg)\delta_{kl} \nonumber\\
	&\quad+\bigg(\frac{1}{4N}(1+(-1)^k)+\frac{8 c_4 }{N} \delta_{k4}+\frac{2(\sigma^2-2)}{N}\delta_{k2}\bigg)^2\frac{k+2}{N^2}\delta_{kl} \nonumber\\
	&\quad +\frac{k+2}{N^2}\bigg( \frac{ k}{2N^2}+\frac{8c_4}{N^2}\delta_{k 2}+\frac{\sigma^2-2}{4N^2}\delta_{k1}\bigg)\delta_{kl}+O_{\prec}(klN^{-4-\zeta})\,,
	\end{align*}
	as well as
	\begin{equation*}
	\mathrm{Cov} \big( \widehat{t}_k({H})\widehat{t}_{k}({H}),\widehat{t}_l({H})\widehat{t}_{l+2}({H})\big) \prec kl N^{-4-\zeta} \,.
	\end{equation*}
\end{pro}
Furthermore, we also need the following proposition on the estimate of the expectations.
\begin{pro}\label{pro.exp}  For any  $k\in\llbracket 1, n_\omega\rrbracket$, we have 
	\begin{align*}
	\mathbb{E} (\widehat{t}_k({H}))^2=&\frac{ k}{2N^2}+\frac{8c_4}{N^2}\delta_{k 2}+\frac{\sigma^2-2}{4N^2}\delta_{k1}+\bigg(\frac{1}{4N}(1+(-1)^k)+\frac{8 c_4 }{N} \delta_{k4}+\frac{\sigma^2-2}{2N}\delta_{k2}\bigg)^2 +O_{\prec}(kN^{-2-\zeta})
	\end{align*}
	and
	\begin{align*}
	\mathbb{E} \widehat{t}_k({H})\widehat{t}_{k+2}({H})=&\bigg(\frac{1}{4N}(1+(-1)^k)+\frac{8 c_4 }{N} \delta_{k4}+\frac{\sigma^2-2}{2N}\delta_{k2}\bigg) \bigg(\frac{1}{4N}(1+(-1)^k)+\frac{8 c_4 }{N} \delta_{k2}\bigg)+O_{\prec} (kN^{-2-\zeta}) \,.
	\end{align*}
\end{pro}

With the discussion above, we can now  prove Theorem \ref{thm. main result}. 
\begin{proof}[Proof of Theorem \ref{thm. main result}]
In the sequel, we state the proof of the real case only. For the complex case, the proof is analogous. 
Fix an even integer $M\geq 4$. Recall the notation $\langle X \rangle\deq X-\bb E X$. We have
\begin{multline} \label{3.13}
N^2\mathcal{A}_{N,\omega} =N^2\bigg(\frac{2}{\pi^2}\sum_{k=1}^M \frac{1}{k^2}  r_{\omega}^{2k}\la\big(\widehat{t}_k({H})\big)^2\ra- \frac{2}{\pi^2}\sum_{k=1}^M \frac{1}{k(k+2)}r_{\omega}^{2k+2}\la \widehat{t}_k({H})\widehat{t}_{k+2}({H})\ra+\frac{1}{\pi^2}r_{\omega}^2 \la\big(\widehat{t}_1({H})\big)^2\ra\bigg)\\
+N^2\bigg(\frac{2}{\pi^2}\sum_{k=M+1}^{n_{\omega}} \frac{1}{k^2} r_{\omega}^{2k} \la\big(\widehat{t}_k({H})\big)^2\ra- \frac{2}{\pi^2}\sum_{k=M+1}^{n_{\omega}}  \frac{1}{k(k+2)}r_{\omega}^{2k+2}\la \widehat{t}_k({H})\widehat{t}_{k+2}({H})\ra\bigg)\\
+\bb E N^2 \bigg(\frac{2}{\pi^2}\sum_{k=1}^{n_\omega} \frac{1}{k^2} r_{\omega}^{2k} \big(\widehat{t}_k({H})\big)^2- \frac{2}{\pi^2}\sum_{k=1}^{n_\omega} r_{\omega}^{2k+2}\frac{1}{k(k+2)}\widehat{t}_k({H})\widehat{t}_{k+2}({H})+\frac{1}{\pi^2}r_{\omega}^{2k}(\widehat{t}_1(H))^2\bigg)\\+O_{\prec}(N^{-2\zeta})
\eqd X^{(1,M)}+X^{(2,M)}+X^{(3)}+O_{\prec}(N^{-2\zeta})\,.
\end{multline}
By Proposition \ref{pro.exp} we have
\begin{multline} \label{eqn: X3}
X^{(3)}=\frac{1}{\pi^2}\sum_{k=1}^{n_\omega}\frac{1}{k}r_{\omega}^{2k}
+\frac{4}{\pi^2}c_4 r_{\omega}^4+\frac{\sigma^2-2}{2\pi^2}r_{\omega}^2
+\frac{1}{2\pi^2}\sum_{n=1}^{\floor{n_{\omega}/2}} \frac{1}{4n^2}r_{\omega}^{4n}
+\frac{8}{\pi^2}c_4^2 r_{\omega}^8 \\
+\frac{(\sigma^2-2)^2}{8\pi^2}r_{\omega}^4
+\frac{1}{\pi^2}c_4 r_{\omega}^8
+\frac{\sigma^2-2}{4\pi^2}r_{\omega}^4
-\frac{1}{2\pi^2}\sum_{n=1}^{\floor{n_{\omega}/2}}\frac{1}{4n(n+1)}r_{\omega}^{4n+2}
-\frac{1}{\pi^2}c_4 r_{\omega}^6
-\frac{1}{3\pi^2}c_4r_{\omega}^{10} \\
-\frac{\sigma^2-2}{16\pi^2}r_{\omega}^6
-\frac{\sigma^2-2}{\pi^2}c_4r_{\omega}^6
+\frac{1}{2\pi^2}r_{\omega}^2+\frac{\sigma^2-2}{4\pi^2}r_{\omega}^2
+O_{\prec}(N^{-\zeta})\\
=-\frac{\log (2\omega)}{\pi^2}+\frac{4}{\pi^2}c_4
+\frac{\sigma^2-2}{2\pi^2}
+\frac{1}{48}
+\frac{8}{\pi^2}c_4^2
+\frac{(\sigma^2-2)^2}{8\pi^2}
+\frac{1}{\pi^2}c_4
+\frac{\sigma^2-2}{4\pi^2}
-\frac{1}{8\pi^2}
-\frac{1}{\pi^2}c_4\\
-\frac{1}{3\pi^2}c_4
-\frac{\sigma^2-2}{16\pi^2}
-\frac{\sigma^2-2}{\pi^2}c_4
+\frac{1}{2\pi^2}+\frac{\sigma^2-2}{4\pi^2} 
+O_{\prec}(N^{-\zeta})
\\=\frac{\alpha \log N}{\pi^2}+b_1+O_{\prec}(N^{-\zeta})\,,
\end{multline}
where $b_1$ is defined in Theorem \ref{thm. main result}. 

Let $(Z_i)_{i \in \bb N^*}$ be independent standard real Gaussian random variables, and by Theorem \ref{thm.by} we see that
\begin{multline*}
X^{(1,M)}\overset{d}{\longrightarrow}
\frac{2}{\pi^2}\sum_{k=1}^{M} \frac{1}{k^2}\Big(\frac{k}{2}+\frac{\sigma^2-2}{4}\delta_{k1}+8c_4\delta_{k2}\Big) (Z_k^2-1)\\
+\frac{4}{\pi^2}\sum_{k=1}^{M}\frac{1}{k^2}\Big(\frac{1}{4}(1+(-1)^k)+\frac{1}{2}(\sigma^2-2)\delta_{k2}+8c_4\delta_{k4}\Big)\Big(\frac{k}{2}+\frac{\sigma^2-2}{4}\delta_{k1}+8c_4\delta_{k2}\Big)^{1/2}Z_k\\
-\frac{2}{\pi^2}\sum_{k=1}^M\frac{1}{k(k+2)}\Big(\frac{k}{2}+\frac{\sigma^2-2}{4}\delta_{k1}+8c_4\delta_{k2}\Big)^{1/2}\Big(\frac{k+2}{2}\Big)^{1/2}Z_kZ_{k+2}\\
-\frac{2}{\pi^2}\sum_{k=1}^M\frac{1}{k(k+2)}\Big(\frac{1}{4}(1+(-1)^k)+\frac{\sigma^2-2}{ 2}\delta_{k2}+8c_4\delta_{k4}\Big)\Big(\frac{k+2}{2}\Big)^{1/2}Z_{k+2}\\
-\frac{2}{\pi^2}\sum_{k=1}^M\frac{1}{k(k+2)}\Big(\frac{1}{4}(1+(-1)^{k+2})+8c_4\delta_{k2}\Big)\Big(\frac{k}{2}+\frac{\sigma^2-2}{4}\delta_{k1}+8c_4\delta_{k2}\Big)^{1/2}Z_k+\frac{1}{\pi^2}\Big(\frac{1}{2}+\frac{\sigma^2-2}{4}\Big)(Z_1^2-1)\\
=\frac{1}{\pi^2}\sum_{k=1}^M\bigg( \frac{1}{ k}(Z_k^2-1)-\frac{1}{\sqrt{k(k+2)}}Z_kZ_{k+2}\bigg)+\frac{1}{\pi^2}\sum_{n=1}^{M/2}\Big(\frac{n+2}{4n^{3/2}(n+1)}Z_{2n}-\frac{1}{4n\sqrt{n+1}}Z_{2n+2}\Big)\\
+\frac{1}{8\pi^2}\big((4\sigma^2-16c_4-5)(1+8c_4)^{1/2}-3\big)Z_2+\frac{2\sqrt{2}}{\pi^2}\Big(c_4 -\frac{\sigma^2-2}{ 16\black}\Big)Z_4-\frac{2}{\sqrt{3}\pi^2}c_4 Z_6\\
+\frac{1}{ \pi^2}\Big(\frac{3(\sigma^2-2)}{4}+\frac{1}{2}\Big)(Z_1^2-1)\black-\frac{1}{2\sqrt{3}\pi^2}\big((2(\sigma^2-2)+4)^{1/2}-2\big)Z_1Z_3\\
+\frac{4}{\pi^2}c_4(Z_2^2-1)-\frac{1}{2\sqrt{2}\pi^2}\big((1+8c_4)^{1/2}-1\big)Z_2Z_4\eqd Y^{(M)}+a_1
\end{multline*} 
as $N \to \infty$. Here we recall the definition of $a_1$ in Theorem \ref{thm. main result}. Let us denote
\[
Y\deq\frac{1}{\pi^2}\sum_{k=1}^\infty\bigg( \frac{1}{ k}(Z_k^2-1)-\frac{1}{\sqrt{k(k+2)}}Z_kZ_{k+2}\bigg)+\frac{1}{\pi^2}\sum_{n=1}^\infty\Big(\frac{n+2}{4n^{3/2}(n+1)}Z_{2n}-\frac{1}{4n\sqrt{n+1}}Z_{2n+2}\Big)\,,
\] and it is easy to see that 
\[
\bb E \big[\big|Y-Y^{(M)}\big|^2\big]=\var(Y-Y^{(M)} )\leq CM^{-1}
\]
for some constant $C>0$ independent of $N,M$. By Proposition \ref{pro.cov} we have
\[
\bb E \big[\big|X^{(2,M)}\big|^2\big]=\var(X^{(2,M)})\leq CM^{-1}\,.
\]
Thus for any fixed $t>0$,
\begin{align*}
&\big|\lim _{N \to \infty}\bb E (\exp (\ii t (X^{(1,M)}+X^{(2,M)})))-\bb E (\exp(\ii t (Y+a_1)))\big|\\
&\leq \lim_{N \to \infty}\big|\bb E (\exp (\ii t (X^{(1,M)}+X^{(2,M)})))-\bb E(\exp (\ii t (X^{(1,M)})))| \\
&\qquad+|\bb E (\exp(\ii t (Y^{(M)}+a_1)))-\bb E (\exp(\ii t (Y+a_1)))|\\
&\leq t\, \bb E [|X^{(2,M)}|] + t\,\bb E[|Y-Y^{(M)}|]\leq 2t \sqrt{\frac{C}{M}}\,.
\end{align*}
Since $X^{(1,M)}+X^{(2,M)}$ is independent of $M$, we have
\begin{equation} \label{3.15}
X^{(1,M)}+X^{(2,M)} \overset{d}{\longrightarrow} Y+a_1
\end{equation}
as $N \to \infty$. We conclude the proof of Theorem \ref{thm. main result} by combining \eqref{3.13} -- \eqref{3.15}.
\end{proof}

\section{Proof of Propositions \ref{pro.cov} and \ref{pro.exp}} \label{sec 4}
In this section, we prove \eqref{19012420} in detail, based on Proposition \ref{pro.19052401}. The other statements in Proposition \ref{pro.cov} and \ref{pro.exp} can be proved in the same manner with the aid of Proposition \ref{pro.19052401}, and thus we omit the details.  We shall rewrite all quantities in terms of the Green function, and then proceed the proof using some estimates of the 4-points and 2-points correlation functions of the Green functions. 

Let $k,l \in \llbracket 1, n_\omega\rrbracket$, and without loss of generality assume $k \geq l$. Recall the definition of $\widehat{\gamma}_{a,k}$ in \eqref{3.4} for $a>0$. According to \eqref{def. t hat}, we can write 
\begin{multline} \label{123}
\mathrm{Cov} \Big(\widehat{t}_k(H)\widehat{t}_k(H), \widehat{t}_l(H)\widehat{t}_l(H)\Big) 
=\frac{1}{16\pi^4}\oint_{\widehat{\gamma}_{1,k}} \oint_{\widehat{\gamma}_{2,k}}\oint_{\widehat{\gamma}_{3,l}}\oint_{\widehat{\gamma}_{4,l}}{T}_k(z_1){T}_k(z_2){T}_l(z_3){T}_l(z_4)\\
\qquad \times \text{Cov}\Big(m_N^\Delta(z_1)m_N^\Delta(z_2), m_N^\Delta(z_3)m_N^\Delta(z_4)\Big)\prod_{i=1}^4 {\rm d}z_i\,.
\end{multline}
Here we choose $a$ to be $1,2,3,4$ for four contours respectively such that they are well separated in case $k\geq l$. 
Analogously,  we can write 
\begin{align*}
\mathbb{E} \widehat{t}_k(H)\widehat{t}_k(H)= -\frac{1}{4\pi^2}\oint_{\widehat{\gamma}_{1,k}}\oint_{\widehat{\gamma}_{2,k}} T_k(z_1)T_k(z_2) \mathbb{E}\big(m_N^\Delta(z_1)m_N^\Delta(z_2)\big) {\rm d}z_1{\rm d}z_2\,.
\end{align*}
In the following we shall abbreviate $\gamma_i \deq {\gamma}_{i,k}$, $\widehat{\gamma}_i\deq \widehat{\gamma}_{i,k}$ for $i=1,2$ and $\gamma_i \deq {\gamma}_{i,l}$, $\widehat{\gamma}_i\deq \widehat{\gamma}_{i,l}$ for $i=3,4$.  Also, we write $z_i=E_i+\ii \eta_i$ for $i=1,2,3,4$. 
Let us define 
\begin{equation} \label{u_j}
u_i=\sqrt{|E^2_i-1|+|\eta_i|}
\end{equation}
for $i=1,...,4$. We set
\begin{align}
t_{i,j}\deq \frac{1}{u_i^2u_j^2}+\frac{1}{u_i^2|z_i-z_j|}+ \frac{1}{u_j^2|z_i-z_j|} \label{def of t_ij}
\end{align}
for $i\ne j\in \{1,2,3,4\}$. For the rest of the paper we will often encounter the following fundamental error 
\begin{equation} \label{error}
\cal E \deq \frac{kl(t_{1,2}t_{3,4}+t_{1,3}t_{2,4}+t_{1,4}t_{2,3})}{N^{4+\zeta}}\,,
\end{equation}
which will be used to bound various error terms. 
The following lemma states an elementary estimate concerning $t_{i,j}$, whose proof is omitted due to its triviality.
\begin{lem} \label{lem:5.22}
	For $i \ne j \in \{1,...,4\}$, we have
	\begin{equation}
	\oint_{\widehat{\gamma}_{i}} \oint_{\widehat{\gamma}_{j}} t_{i,j}\, {\rm d}z_i\, \dd z_j\prec 1 \,.
	\end{equation}
\end{lem}
Next, for $z_i\in \gamma_i, z_j \in \gamma_j, i \ne j$, we define the following two functions
\begin{align} \label{5.18}
f(z_i,z_j)\deq &\frac{(z_i-z_j)^2-(\sqrt{z_i^2-1}-\sqrt{z_j^2-1})^2}{2N^2(z_i-z_j)^2\sqrt{z_i^2-1}\sqrt{z_j^2-1}}
-\frac{2}{N^2}c_4 m(z_i)m'(z_i)m(z_j)m'(z_j)-\frac{\sigma^2-2}{4N^2}m'(z_i)m'(z_j)\,,
\end{align}
and
\begin{equation} \label{eqn: g}
g(z_i)\deq -\frac{1}{4N\sqrt{z_i^2-1}}\Big(m'(z_i)+4\,c_4(m(z_i))^4+(\sigma^2-2)(m(z_i))^2\Big)\,.
\end{equation}
With the notations defined in (\ref{def of t_ij}), (\ref{error}), (\ref {5.18}) and  (\ref{eqn: g}), we can state  our main technical estimate as the following proposition, whose proof will be postponed to Section \ref{sec 5}.
\begin{pro}\label{pro.19052401} Let $z_i\in \widehat{\gamma}_{i}$ for $i=1,...,4$, we have 
	\begin{equation} \label{19052410}
	\mathbb{E} \Big(m_N^\Delta(z_1)m_N^\Delta(z_2)\Big)= f(z_1,z_2)+g(z_1)g(z_2)+O_{\prec}(kt_{1,2}N^{-2-\zeta})\,.
	\end{equation}
	and
	\begin{multline} \label{19052411}
	\mathrm{Cov}\Big(m_N^\Delta(z_1)m_N^\Delta(z_2), m_N^\Delta(z_3)m_N^\Delta(z_4)\Big)\\
	=f(z_1,z_3)f(z_2,z_4)+f(z_1,z_4)f(z_2,z_3)+g(z_1)g(z_3)f(z_2,z_4) +g(z_1)g(z_4)f(z_2,z_3)\\
	+g(z_2)g(z_3)f(z_1,z_4)+g(z_2)g(z_4)f(z_1,z_3)+O_{\prec}(\cal E)\eqd F(z_1,z_2,z_3,z_4) +O_{\prec}(\cal E)\,.
	\end{multline}
\end{pro}

Based on Proposition \ref{pro.19052401}, we can now prove \eqref{19012420}.

\begin{proof}[Proof of \eqref{19012420}]
	By \eqref{eqn: bound T} and Lemma \ref{lem:5.22}, we have
	\begin{equation*}
	\oint_{\widehat{\gamma}_{1}} \oint_{\widehat{\gamma}_{2}}\oint_{\widehat{\gamma}_{3}}\oint_{\widehat{\gamma}_{4}} \big|{T}_k(z_1){T}_k(z_2){T}_l(z_3){T}_l(z_4)\big|\,\cal E\,\prod_{i=1}^4 {\rm d}z_i\prec \frac{kl}{N^{4+\zeta}}\,.
	\end{equation*}
 This together with \eqref{123} and \eqref{19052411} leads to
\begin{align}  \label{4.8}
&\mathrm{Cov} \Big(\widehat{t}_k(H)\widehat{t}_k(H), \widehat{t}_l(H)\widehat{t}_l(H)\Big) \\
=&\frac{1}{16\pi^4}\oint_{\widehat{\gamma}_{1}} \oint_{\widehat{\gamma}_{2}}\oint_{\widehat{\gamma}_{3}}\oint_{\widehat{\gamma}_{4}}{T}_k(z_1){T}_k(z_2){T}_l(z_3){T}_l(z_4)
 F(z_1,z_2,z_3,z_4)\prod_{i=1}^4 {\rm d}z_i
+O_{\prec}(klN^{-(4+\zeta)})\nonumber\\
=&\frac{1}{16\pi^4}\oint_{\gamma_{1}} \oint_{\gamma_{2}}\oint_{\gamma_{3}}\oint_{\gamma_{4}}{T}_k(z_1){T}_k(z_2){T}_l(z_3){T}_l(z_4) F(z_1,z_2,z_3,z_4)\prod_{i=1}^4 {\rm d}z_i
+O_{\prec}(klN^{-(4+\zeta)})\nonumber\,,
\end{align}
where in the second step we used \eqref{3.4} and \eqref{eqn: bound T} to replace the domain of integration $\oint_{\widehat{\gamma}_{1}} \oint_{\widehat{\gamma}_{2}}\oint_{\widehat{\gamma}_{3}}\oint_{\widehat{\gamma}_{4}}$ by
$\oint_{{\gamma}_{1}} \oint_{{\gamma}_{2}}\oint_{{\gamma}_{3}}\oint_{{\gamma}_{4}}$, with an error of  $O(N^{-10})$ abosorbed by $O_{\prec}(klN^{-(4+\zeta)})$. 

Let us denote by $q_i=m(z_i)/2$ for $i=1,...,4$. To compute the RHS of \eqref{4.8}, we first consider
	\begin{equation} \label{4.9}
	\oint_{\gamma_{1}}\oint_{\gamma_{3}}{T}_k(z_1){T}_l(z_3) f(z_1,z_3) {\rm d}z_1 {\rm d}z_3=\oint_{|q_1|=\rho_1}\oint_{|q_3|=\rho_3}{T}_k(z_1){T}_l(z_3) f(z_1,z_3) {\rm d}z_1 {\rm d}z_3
	\end{equation} 
	for some $0<\rho_1<\rho_3<1$. Here we used the assumption $k \geq l$, and the fact that ${T}_k(z_1){T}_l(z_3) f(z_1,z_3)$ is analytic for $z_1,z_3 \in \bb C\backslash[-1,1]$. Let $q_i'$ denote the derivative of $q_i$ with respect to $z_i$.	By
	\[
	z_i=\frac{-q_i-q_i^{-1}}{2}\,, \quad \sqrt{z_i^2-1}=\frac{q_i-q_i^{-1}}{2}\,, \quad \mbox{and} \quad  q_i'=-\frac{q_i}{\sqrt{z_i^2-4}}\,,
	\]
	we can write
	\begin{align*}
	&f(z_1,z_3)
	=\frac{q'_1q'_3}{2N^2q_1q_3}\Big(1-\Big(\frac{1+q_1q_3}{1-q_1a_3}\Big)^2\Big)-\frac{32}{N^2}c_4 q_1q_1'q_3q_3' -\frac{\sigma^2-2}{N^2}q_1'q_3',\nonumber\\
&T_k(z_1)=\frac{1}{2}(q_1^k+q_1^{-k}).
	\end{align*}
	Thus
	\begin{multline*}
	\eqref{4.9}=\oint_{|q_1|=\rho_1}\oint_{|q_3|=\rho_3} \frac{1}{4}\big(q_1^k+q_1^{-k}\big)(q_3^l+q_3^{-l}) \\ \times
	\bigg(\frac{1}{2N^2q_1q_3}\Big(1-\Big(\frac{1+q_1q_3}{1-q_1q_3}\Big)^2\Big)-\frac{32}{N^2}c_4q_1q_3-\frac{\sigma^2-2}{N^2}\bigg){\rm d}q_1{\rm d}q_3 \eqd (\mathrm{I})+(\mathrm{II})+(\mathrm{III})\,.
	\end{multline*}
	By writing $q_i=\rho_i\e{\ii \theta_i}$ and using $\rho_1,\rho_3\in (0,1)$, we have
	\begin{align*}
	(\mathrm{I})=&\int_{-\pi}^{\pi}\int_{-\pi}^{\pi} \frac{1}{4}\big(\rho_1^k\e{\ii k\theta_1} +\rho_1^{-k}\e{-\ii k \theta_1}\big)\big(\rho_3^l\e{\ii l\theta_3} +\rho_3^{-l}\e{-\ii l \theta_3}\big)\\
	&\times \bigg(\frac{\ii^2}{2N^2}\bigg(1-\bigg(\frac{1+\rho_1\rho_3\text{e}^{\ii(\theta_1+\theta_3)}}{1-\rho_1\rho_3\text{e}^{\ii(\theta_1+\theta_3)}}\bigg)^2\bigg)\bigg)\dd \theta_1 \dd \theta_3\\
	=&\int_{-\pi}^{\pi}\int_{-\pi}^{\pi} \frac{1}{4}\big(\rho_1^k\e{\ii k\theta_1} +\rho_1^{-k}\e{-\ii k \theta_1}\big)\big(\rho_3^l\e{\ii l\theta_3} +\rho_3^{-l}\e{-\ii l \theta_3}\big)\\
	&\times\frac{-1}{2N^2}(-4)\sum_{n=1}^\infty n (\rho_1\rho_3)^{n}\text{e}^{ \ii n (\theta_1+\theta_3)}\dd \theta_1 \dd \theta_3
	=\frac{2\pi^2 k}{N^2}\delta_{kl} \,.
	\end{align*}
	Similarly, we can show that 
	$$
	(\mathrm{II})=\frac{32\pi^2}{N^2}c_4\delta_{k2}\delta_{l2}\,, \quad \mbox{and} \quad (\mathrm{III})=\frac{(\sigma^2-2)\pi^2}{N^2}\delta_{k1}\,,
	$$
	which imply
	\[
	\oint_{\gamma_{1}}\oint_{\gamma_{3}}{T}_k(z_1){T}_l(z_3) f(z_1,z_3) {\rm d}z_1 {\rm d}z_3=\frac{2\pi^2 k}{N^2}\delta_{kl}+\frac{32\pi^2}{N^2}c_4\delta_{k2}\delta_{l 2}+\frac{(\sigma^2-2)\pi^2}{N^2}\delta_{l1}\delta_{k1}\,.
	\]
	Similarly, we have
	\[
	g(z_1)=\frac{q_1q_1'}{N(1-q_1^2)}+\frac{16c_4q_1^3q_1'}{N}+\frac{(\sigma^2-2)q_1q_1'}{N}\,.
	\]
	Therefore, 
	\begin{multline*}
	\oint_{\gamma_1} T_k(z_1) g(z_1) \dd z_1=\oint_{|q_1|=\rho_1}\frac{1}{2}(q_1^k+q_1^{-k})\Big(\frac{q_1}{N(1-q_1^2)}+\frac{16c_4q_1^3}{N}+\frac{(\sigma^2-2)q_1}{N}\Big) \dd q_1\\
	=\int_{-\pi}^{\pi} \frac{1}{2}(q_1^k+q_1^{-k})\Big(\frac{1}{N}\sum_{n=1}^{\infty}q_1^{2n-1}+\frac{16c_4q_1^3}{N}+\frac{(\sigma^2-2)q_1}{N}\Big) \ii q_1 \dd \theta_1\\=\frac{\pi \ii}{2N}(1+(-1)^k)+\frac{16\pi c_4 \ii}{N} \delta_{k4}+\frac{\pi (\sigma^2-2)\ii}{N}\delta_{k2}\,,
	\end{multline*}
	where in the second step we used the change of variable $q_1=\rho_1 \e{\ii \theta }$. Plugging the above into \eqref{4.8} we have \eqref{19012420} as desired.
\end{proof}

\section{Proof of Proposition \ref{pro.19052401}} \label{sec 5}
In this section, we prove Proposition \ref{pro.19052401}. We will  state the proof of (\ref{19052411}) in detail. The proof of (\ref{19052410}) is similar and simpler, and thus we omit the details. Recall  $z_i=E_i+\ii \eta_i$ for $i=1,2,...,4$. For simplicity, we denote by 
\begin{align*}
G=G(z_1), \quad S=G(z_2), \quad T=G(z_3), \quad V=G(z_4). 
\end{align*}
In sections \ref{sec 5.1} -- \ref{sec: 5.3}, we shall first prove the following  estimate for the centered quantities.
\begin{pro} \label{prop: centered}
	Let $z_i\in \widehat{\gamma}_{i}$ for $i=1,...,4$, we have 
	\begin{align}
	\cov \big(\langle \underline{G}\rangle \langle \underline{S}\rangle, \langle \underline{T}\rangle \langle \underline{V}\rangle\big)=f(z_1,z_3)f(z_2,z_4)+f(z_1,z_4)f(z_2,z_3)+O_{\prec}(\cal E)\,, \label{19052450}
	\end{align} 
	where $f(\cdot,\cdot)$ is defined  in \eqref{5.18} and $\cal E$ is defined in (\ref{error}). 
\end{pro}

We emphasize here that all the Green functions have deterministic upper bound $N^{10}$ in operator norm in this section, since we are working on $\widehat{\gamma}_{i}$'s (c.f. (\ref{3.4})). According to Lemma \ref{prop_prec} (ii), it will be clear that all the high probability bounds on the functionals of Green functions  in this section will  be still valid after  one takes expectation of the functionals.   

\subsection{The first step} \label{sec 5.1} To study the LHS of (\ref{19052450}),  it suffices to estimate the 4-point and 2-point correlation functions of the Green functions 
\begin{align*}
\mathbb{E}\langle \underline{G}\rangle \langle \underline{S}\rangle \langle \underline{T}\rangle \langle \underline{V}\rangle,  \qquad  \mathbb{E}\langle \underline{G}\rangle \langle \underline{S}\rangle. 
\end{align*}
By the resolvent identity 
\begin{equation*}
zG=HG-I\,,
\end{equation*}
we have
\begin{equation*}
z_1 \mathbb{E}\langle \underline{G}\rangle \langle \underline{S}\rangle \langle \underline{T}\rangle \langle \underline{V}\rangle =\frac{1}{N}\sum_{i,j} \mathbb{E} H_{ij} G_{ji}\big\la\langle \underline{S}\rangle \langle \underline{T}\rangle \langle \underline{V}\rangle\big\ra\,.
\end{equation*}
We compute the RHS of the above using Lemma \ref{lem:cumulant_expansion} with $\ell=3$, $h=H_{ij}$, $f\equiv f_{ij}=G_{ji}\big\la\langle \underline{S}\rangle \langle \underline{T}\rangle \langle \underline{V}\rangle\big\ra$, and get
\begin{align}
z_1 \mathbb{E}\langle \underline{G}\rangle \langle \underline{S}\rangle \langle \underline{T}\rangle \langle \underline{V}\rangle=\sum_{n=1}^{3}\frac{1}{n!}\frac{1}{N}\sum_{i,j}\cal C_{n+1}(H_{ij}) \bb E\partial^n_{ij} \big(G_{ji}\big\la\langle \underline{S}\rangle \langle \underline{T}\rangle \langle \underline{V}\rangle\big\ra\big)+\cal R \eqd \sum_{n=1}^3L_n+\cal R\,,
\label{20070601}
\end{align}
where we abbreviate $\partial_{ij}^n\deq \frac{\partial^n }{\partial H_{ij}^n}$, and $\cal R$ is the remainder term satisfying
\begin{equation} \label{R}
\begin{aligned}	
\cal R& \prec \frac{1}{N}\sum_{i,j} \bigg(\E  \sup_{|x| \le |H_{ij}|}\big| \partial_{ij}^{4}f({H}^{ij}+x\Delta^{ij})\big|^2\cdot\E\big| {H^{10}_{ij}}\mathbf{1}(|H_{ij}|>t)\big|\bigg)^{1/2}\\&\ \ +\frac{1}{N}\sum_{i,j} \E |H_{ij}|^{5} \cdot  \E { \sup_{|x| \le t}\big| \partial_{ij}^{4}f({H}^{ij}+x\Delta^{ij})\big|}
\end{aligned}
\end{equation}
for any $t>0$. 
Here  we define $\Delta^{ij}\in \bb C^{N\times N}$ such that $\Delta^{ij}_{xy}=(\delta_{xi}\delta_{jy}+\delta_{xj}\delta_{iy})(1+\delta_{xy})^{-1}$, and  ${H}^{ij}\deq H- H_{ij}\Delta^{ij}$. In other words, $H^{ij}$ is obtained from $H$ by setting both  $(i,j)$ and
$(j,i)$ entries to 0. Note that
\begin{multline*}
L_1=\frac{1}{N}\sum_{i,j}\frac{1+\delta_{ij}}{4N}\bb E (\partial_{ij}G_{ij})\big\la\langle \underline{S}\rangle \langle \underline{T}\rangle \langle \underline{V}\rangle\big\ra+\frac{1}{N}\sum_{i,j}\frac{1+\delta_{ij}}{4N}\bb E G_{ij}\partial_{ij}\big\la\langle \underline{S}\rangle \langle \underline{T}\rangle \langle \underline{V}\rangle\big\ra+K
\eqd (A)+(B)+K\,,
\end{multline*}
where
\begin{align}
K\deq \frac{\sigma^2-2}{4N^2}\sum_i\bb E \partial_{ii}\big(G_{ii}\big\la\langle \underline{S}\rangle \langle \underline{T}\rangle \langle \underline{V}\rangle\big\ra\big)\,. \label{20070605}
\end{align} 
By the differential rule 
\begin{equation} \label{eqn:diff}
\partial_{ij}G_{xy} =-(G_{xi}G_{jy}+G_{xj}G_{iy})(1+\delta_{ij})^{-1}\,,
\end{equation}
we have
\begin{multline*}
(A)=-\frac{1}{4N^2}\sum_{i,j}\bb E (G_{ii}G_{jj}+G^2_{ij})\big\la\langle \underline{S}\rangle \langle \underline{T}\rangle \langle \underline{V}\rangle\big\ra\\
=-\frac{1}{4}\bigg(2\bb E \langle \underline{G}\rangle \langle \underline{S}\rangle \langle \underline{T}\rangle \langle \underline{V}\rangle+ \mathbb{E} \la\underline{G}\ra^2 \big\langle \underline{S}\rangle \langle \underline{T}\rangle \langle \underline{V}\rangle -\bb E\la\underline{G}\ra^2\bb E \big\langle \underline{S}\rangle \langle \underline{T}\rangle \langle \underline{V}\rangle+\frac{1}{N}\mathbb{E}\la \underline{G^2}\ra \langle \underline{S}\rangle \langle \underline{T}\rangle \langle \underline{V}\rangle \Big)
\end{multline*}
and 
\[
(B)=-\frac{1}{4}\Big(\frac{2}{N^2}\mathbb{E} \underline{GS^2}\langle \underline{T}\rangle\langle \underline{V}\rangle
+\frac{2}{N^2}\mathbb{E} \underline{GT^2}\langle \underline{S}\rangle \langle \underline{V}\rangle+\frac{2}{N^2}\mathbb{E} \underline{GV^2}\langle \underline{S}\rangle \langle \underline{T}\rangle\Big)\,.
\]
Thus we arrive at
\begin{multline} \label{5.5}
\mathbb{E}\langle \underline{G}\rangle \langle \underline{S}\rangle \langle \underline{T}\rangle \langle \underline{V}\rangle=  \frac{1}{-4s_1}\bigg( \mathbb{E} \la\underline{G}\ra^2 \big\langle \underline{S}\rangle \langle \underline{T}\rangle \langle \underline{V}\rangle -\bb E\la\underline{G}\ra^2\bb E \big\langle \underline{S}\rangle \langle \underline{T}\rangle \langle \underline{V}\rangle+\frac{1}{N}\mathbb{E}\la \underline{G^2}\ra \langle \underline{S}\rangle \langle \underline{T}\rangle \langle \underline{V}\rangle \\ +\frac{2}{N^2}\mathbb{E} \underline{GS^2}\langle \underline{T}\rangle\langle \underline{V}\rangle
+\frac{2}{N^2}\mathbb{E} \underline{GT^2}\langle \underline{S}\rangle \langle \underline{V}\rangle+\frac{2}{N^2}\mathbb{E} \underline{GV^2}\langle \underline{S}\rangle \langle \underline{T}\rangle-4L_2-4L_3-4K-4\cal R\bigg)\,,
\end{multline}
where
$$s_i\deq z_i+\frac{1}{2}\bb E \ul{G}(z_i)$$
for $i=1,...,4$. On the RHS of \eqref{5.5}, the first three terms and $L_2/s_1$, $\cal R/s_1 $ are the error terms, while other terms contain the leading contributions. The analysis of the error terms in \eqref{5.5} is broken down into the estimates in the following section. 

\subsection{The estimates} \label{sec:estimate}
We begin with some preliminary estimates on Green functions.  Recall $u_i$, $i=1,...,4$ defined  in \eqref{u_j} and also the notations introduced in (\ref{notation for quadratic form}).  In this section, most estimates are explicitly  stated for $z_1\in \widehat{\gamma}_1$ only for convenience, but they also hold if we replace  $(z_1,\widehat{\gamma}_1)$ by any other $(z_i,\widehat{\gamma}_i), i=2,3,4$. 
\begin{lem} \label{lem for s_1}
	Let $\b w,\b v \in \bb C^{N}$ be deterministic satisfying $\|\b w\|_2=\|\b v \|_2=1$. We have
	\[
 	\frac{1}{|s_1|} \asymp \frac{1}{u_1} \leq k\,, \quad |\ul{G}(z_1)-m(z_1)| \prec \frac{k}{Nu_1}\,, \quad  |G_{\b w\b v}(z_1)-m(z_1)\langle \b w, \b v\rangle| \prec \sqrt{\frac{k}{Nu_1}}\,, \quad |G_{\b w \b v}(z_1)| \prec 1
	\]
	and 
	\[
	|(G^2)_{\b w \b v}(z_1)| \prec \frac{1}{u_1^{3/2}}\,, \quad \ul{G^2}(z_1) \prec \frac{1}{u_1}\,, \quad \ul{G^3}(z_1) \prec \frac{\sqrt{N}}{u_1^{3/2}}\,.
	\]
	uniformly for $z_1 \in \widehat{\gamma}_{1}$. 
\end{lem}
\begin{proof}
	The first four relations are simple consequences of Theorem \ref{thm. local semicircle law} and the construction of our contour $\widehat{\gamma}_1$. In order to prove the fifth  estimate, 
	we write
	\begin{equation*}
	G^2=\left(\frac{(H-E)/\eta+\mathrm{i}}{(H-E)^2/\eta^2+1}\right)^2 \cdot \eta^{-2} = q\bigg(\frac{H-E}{\eta}\bigg) \cdot \eta^{-2}\,,
	\end{equation*}
	where we defined $q(x) \deq \big(\frac{x + \ii}{x^2 + 1}\big)^2$. Note that $q: \bb R \to \bb C$ is smooth, and for any $n \in \bb N$, $|q^{(n)}(x)|=O((1 + |x|)^{-2})$.  We write $q_{\eta}(x)\deq q\big(\frac{x-E}{\eta}\big)$, and define $\tilde q_{\eta}$ to be the almost analytic extension of $q_\eta$ (see also \eqref{eyaya}). Let $\xi\in C_c^\infty(\mathbb{R}) $ be a fixed ($N$-independent)  smooth cutoff function which satisfies  $\xi(0)=1$.  Set $
	\chi(z)\deq\xi(y/\eta)$, where $z= x+\ii y$. Then, applying Helffer-Sj{\"o}strand formula in  Lemma \ref{HS}, we have 
	$$
	q_{\eta}(H)= \frac{1}{\pi}\int_{\bb C}\frac{\partial_{\bar{z}}(\tilde{q}_{\eta}(z)\chi(z))}{H-z}\,\mathrm{d}^2z,
	$$
due to the arbitrariness of the cutoff $\chi(z)$.  	
Therefore, we have 
	\begin{multline*}
	 (G^2)_{\b w\b v}(z)-m'(z)\la \b w,\b v \ra\\
	 =\frac{1}{2\pi\eta^2}\int_{\bb R^2} \bigg(\ii yq^{\prime\prime}_{\eta}(x)\xi(y/\eta)+\frac{\ii}{\eta} q_{\eta}(x) \xi^{\prime}(y/\eta)-\frac{y}{\eta}q^{\prime}_{\eta}(x)\xi^{\prime}(y/\eta)\bigg)
	 \big(G_{\b w\b v}(x+\ii y)-m(x+\ii y)\la \b w,\b v \ra  \big)\,\dd x  \dd y\,.
	\end{multline*}
	 By the above equation and Theorem \ref{thm. local semicircle law},  it can be shown (see e.g.\ \cite[Lemma 4.4]{HK16}) that 
	\[
	(G^2)_{\b w\b v}(z_1)\prec |m'(z_1)\la \b w,\b v \ra | + \frac{\sqrt{\frac{1}{N\eta_1}}}{|\eta_1|} \prec \frac{1}{u_1^{3/2}}
	\]
	when $|E_1|\leq 1+1/(2k^2)$. On the other hand, \eqref{19083101} and $k \leq n_{\omega}$ imply 
	\begin{multline} \label{guonianle}
	(G^2)_{\b w \b v}(z_1) =\sum_i G_{\b wi}G_{i \b v} \\
	\prec \sum_i |\la \b w, \b e_i\ra \la \b e_i,\b v \ra |+\sum_i \big(|\la \b w,\b e_i \ra| +|\la \b e_i,\b v \ra |\big)\frac{1}{\sqrt{N}(|E_1^2-1|+|\eta_1|)^{1/4}}+\frac{1}{\sqrt{|E_1^2-1|+|\eta_1|}}\prec \frac{1}{u_1}
	\end{multline}
	when $|E_1|>1+1/(2k^2)$. This prove the fifth estimate. By Theorem \ref{thm. local semicircle law} and Lemma \ref{HS}, we see that
	\[
	 \ul{G^2}(z_1) \prec |m'(z_1)|+\frac{1}{N|\eta_1|^2} \prec 
	\frac{1}{u_1}
	\]
	when $|E_1|\leq 1+1/(2k^2)$. Together with \eqref{guonianle} we deduce the sixth estimate. The proof of the last relation follows in a similar fashion and we omit the details.
\end{proof}

In the sequel, we will show that the first three terms, $L_2/s_1$, $\cal R/s_1 $ on the RHS of \eqref{5.5} are small. The results are stated in Lemmas \ref{lem:5.2}, \ref{lem:5.4} and \ref{lem:5.5}, followed by the proofs. 
\begin{lem}  \label{lem:5.2}
	Let $\cal R$ be as in \eqref{5.5}. We have
	\[
	R/s_1  \prec \cal E\,.
	\]
\end{lem}

\begin{proof}
	Recall that $\Delta^{ij}\in \bb C^{N\times N}$ such that $\Delta^{ij}_{xy}=(\delta_{xi}\delta_{jy}+\delta_{xj}\delta_{iy})(1+\delta_{xy})^{-1}$, and  ${H}^{ij}\deq H- H_{ij}\Delta^{ij}$.
	Fix $i,j$ and set $\Theta=H_{ij}\Delta^{ij}$, $\widehat{G}\equiv \widehat{G}(z_1)\deq (H^{ij}-z_1)^{-1}$. We omit the $i,j$ dependence from the notations $\Theta$ and $\widehat{G}$ for convenience.   By resolvent expansion we have 
	\[
	\widehat{G}=G+\sum_{k=1}^{99}(G\Theta)^kG+(G\Theta)^{100}\widehat{G}\,.
	\]
	Note that at most two entries of $\Theta$ are nonzero, and they are stochastically dominated by $N^{-1/2}$. Together with the trivial bound $\max_{x,y \in \qq{1,N}}|\widehat{G}_{xy}|\leq N^{10}$ and Lemma \ref{lem for s_1} we have
	\begin{equation} \label{baolan}
	|\ul{\widehat{G}}(z_1)-m(z_1)| \prec |\ul{G}(z_1)-m(z_1)|+N^{-3/2}\max_{x,y \in \qq{1,N}}|(G^2)_{xy}|+N^{-50}\max_{x,y \in \qq{1,N}}|\widehat{G}_{xy}| \prec \frac{k}{Nu_1}
	\end{equation}
as well as
	\begin{equation} \label{A.11}
 \max_{x,y \in \qq{1,N}} |\widehat{G}_{xy}(z_1)-\delta_{xy}m(z_1)| \prec  \max_{x,y \in \qq{1,N}} |{G}_{xy}(z_1)-\delta_{xy}m(z_1)| +N^{-1/2} +N^{-50}\max_{x,y \in \qq{1,N}}|\widehat{G}_{xy}| \prec \sqrt{\frac{k}{Nu_1}}
	\end{equation}
	uniformly for $z_1 \in \widehat{\gamma}_{1,k}$. Using \eqref{A.11} and the fact that $\widehat{G}$ is independent of $H_{ij}$, we have
	\begin{equation} \label{AAA}
	\max_{x,y \in \qq{1,N}}\sup_{|H_{ij}|\leq N^{-1/2+\zeta}} |\widehat{G}_{xy}| \prec 1\,.
	\end{equation} 
Further, we apply the resolvent expansion formula which also says
	\begin{equation} \label{eqn: A.12}
	G=\widehat{G}-(\widehat{G}\Theta)\widehat{G}+(\widehat{G}\Theta)^2G\,.
	\end{equation}
By \eqref{AAA} and \eqref{eqn: A.12}, it is easy to check the bound
	\[
	\max_{x,y \in \qq{1,N}}\sup_{|H_{ij}|\leq N^{-1/2+\zeta}} |G_{xy}| \prec 1+ N^{-1+2\zeta} \max_{x,y \in \qq{1,N}}\sup_{|H_{ij}|\leq N^{-1/2+\zeta}} |G_{xy}|\,,
	\]
	which further implies
	\[
	\max_{x,y \in \qq{1,N}}\sup_{|H_{ij}|\leq N^{-1/2+\zeta}} |G_{xy}| \prec 1\,.
	\]
	Inserting the above into the RHS of \eqref{eqn: A.12} and applying \eqref{baolan}, \eqref{A.11}, we have
	\begin{equation} \label{3.67}
	\max_{x \ne y \in \qq{1,N}}\sup_{|H_{ij}|\leq N^{-1/2+\zeta}}  |{G}_{ x  y}| \prec \sqrt{\frac{k}{Nu_1}}, \quad  \quad \max_{x \in \qq{1,N}} \sup_{|H_{ij}|\leq N^{-1/2+\zeta}}  |{G}_{ x  x}|\prec 1\,,
	\end{equation}
	\begin{equation} \label{5.133}
	\sup_{|H_{ij}|\leq N^{-1/2+\zeta}} |\ul{G}-m(z_1)|\prec \frac{k}{Nu_1}\,,
	\end{equation}
	and
	\begin{equation} \label{5.14}
	\max_{ x,  y \in \qq{1,N}} \sup_{|H_{ij}|\leq N^{-1/2+\zeta}}|(G^2)_{xy}| \prec \frac{k}{u_1}\,.
	\end{equation}
	Similar bounds also hold when $G$ is replaced by $S,T,V$. By setting $t=N^{-1/2+\zeta}$ in \eqref{R}, we see that
	\begin{multline} \label{tau1}
	\cal R/s_1 \prec \cal R /u_1 \prec \frac{1}{Nu_1}\sum_{i,j} \bigg(\E  \sup_{|x| \le |H_{ij}|}\big| \partial_{ij}^{4}f({H}^{ij}+x\Delta^{ij})\big|^2\cdot\E\big| {H^{10}_{ij}}\mathbf{1}(|H_{ij}|>N^{-1/2+\zeta})\big|\bigg)^{1/2}\\ +\frac{1}{Nu_1}\sum_{i,j} \E |H_{ij}|^{5} \cdot  \E { \sup_{|x| \le N^{-1/2+\zeta}}\big| \partial_{ij}^{4}f({H}^{ij}+x\Delta^{ij})\big|}\,,
 	\end{multline}
	where we recall $f\equiv f_{ij}=G_{ji}\big\la\langle \underline{S}\rangle \langle \underline{T}\rangle \langle \underline{V}\rangle\big\ra$. Let us first estimate the second term on the RHS of \eqref{tau1}. By our assumption $\mathbb{E}|H_{ij}|^a=O(N^{-a/2})$,  this term is bounded by
	\begin{equation} \label{rise}
	O_{\prec}(N^{-3/2}u_1^{-1})\cdot \max_{i,j\in \qq{1,N}} \E { \sup_{|x| \le N^{-1/2+\zeta}}\big| \partial_{ij}^{4}f({H}^{ij}+x\Delta^{ij})\big|}\,.
	\end{equation}
Using \eqref{eqn:diff} and \eqref{3.67} -- \eqref{5.14}, we have
\[
\sup_{|H_{ij}|\leq N^{-1/2+\zeta}}|\partial^n_{ij}G_{ij}| \prec 1\,, \quad \sup_{|H_{ij}|\leq N^{-1/2+\zeta}}|\partial^n_{ij}\langle \ul{S} \rangle | \prec \frac{k}{Nu_2}, \quad \quad \sup_{|H_{ij}|\leq N^{-1/2+\zeta}}|\partial^n_{ij}\langle \ul{T} \rangle | \prec \frac{l}{Nu_3}
\]
for all fixed $n \in \bb N$, and the same bound holds if we replace $T$ by $V$ in the last inequality.  These bounds further imply
\[
\sup_{|x| \le N^{-1/2+\zeta}}\big| \partial_{ij}^{4}f({H}^{ij}+x\Delta^{ij})\big| \prec \frac{kl^2}{N^3u_2u_3u_4} \,.
\] 
Using the above bound and Lemma \ref{prop_prec}, we see that \eqref{rise} is bounded by
	\[
	O_{\prec}(N^{-3/2}u_1^{-1})\cdot  \frac{kl^2}{N^3u_2u_3u_4} \prec \frac{kl}{N^{4+\zeta}u_1u_2u_3u_4} \prec \cal E
	\]
	as desired. Further, note that $\max_{i,j}|H_{ij}| \prec  N^{-1/2}$, which by definition implies $\mathbb P(|H_{ij}|>N^{-1/2+\zeta})=O(N^{-D})$ for any fixed $D>0$. Also recall our moment bound $\bb E|H_{ij}|^a=O(N^{-a/2})$ for all $i,j$ and all fixed $a$. By Cauchy-Schwarz inequality we have
	\begin{equation} \label{theshy}
	\bb E \big| {H^{10}_{ij}}\mathbf{1}(|H_{ij}|>N^{-1/2+\zeta})\big| =O(N^{-5-D/2}) 
	\end{equation}
	for any fixed $D>0$. Moreover, as we mentioned earlier, since $z_1 \in \widehat{\gamma}_1$, we have the trivial deterministic bound 
	\[
	\sup_{i,j}|G_{ij}(z_1)| \leq \|G(z_1)\| \leq |\eta_1|^{-1} \leq N^{10}
	\]
which  together with \eqref{eqn:diff}  implies
	\begin{equation} \label{rookie}
	  \sup_{|x| \le |H_{ij}|}\big| \partial_{ij}^{4}f({H}^{ij}+x\Delta^{ij})\big|^2 \leq  \sup_{x \in \bb R}\big| \partial_{ij}^{4}f({H}^{ij}+x\Delta^{ij})\big|^2 =O(N^{1000})\,.
	\end{equation}
	Using  \eqref{rookie} and \eqref{theshy} with sufficiently large $D$, the first term on the RHS of \eqref{tau1} can be easily bounded by $O(N^{-10})$. This finishes the proof.
\end{proof}

\begin{rem} \label{rmk:5.4}
	The method presented in Lemma \ref{lem:5.2} of treating the remainder term was introduced in \cite[Lemma 4.6]{HK16}, and it is generally effective in  estimating the remainder terms from cumulant expansions. In particular, the method applies to $\cal R^{(1)}$, $\cal R^{(2,ij)}$, $\cal R^{(3,ij)}$,$\cal R^{(4)}, \cal R^{(5)}$, and $\cal R^{(6)}$ in the sequel, and we shall omit the details of the estimation for those terms.
\end{rem}

\begin{lem} \label{lem:5.4}
	The first three terms on the RHS of \eqref{5.5} are bounded by $O_{\prec}(\cal E)$.
\end{lem}
\begin{proof}
	We only show the details for the first term on the RHS of \eqref{5.5}. The proof for the other two terms is similar. Applying $z_4V=HV-I$ and the cumulant expansion in Lemma \ref{lem:cumulant_expansion}, we can get
	\begin{multline} \label{5.7}
	\frac{1}{-4s_1}\mathbb{E} \la\underline{G}\ra^2 \big\langle \underline{S}\rangle \langle \underline{T}\rangle \langle \underline{V}\rangle=  \frac{1}{16s_1s_4}\bigg( \mathbb{E} \la\underline{G}\ra^2 \big\langle \underline{S}\rangle \langle \underline{T}\rangle \langle \underline{V}\rangle^2 -\bb E\la\underline{V}\ra^2\bb E  \langle \underline{G}\rangle^2\langle \underline{S}\rangle \langle \underline{T}\rangle\\+\frac{1}{N}\mathbb{E}\la \underline{G}\ra^2 \langle \underline{S}\rangle \langle \underline{T}\rangle \langle \underline{V^2}\rangle  +\frac{4}{N^2}\mathbb{E} \underline{VG^2}\langle\ul{G}\rangle\langle \underline{S}\rangle\langle \underline{T}\rangle
	+\frac{2}{N^2}\mathbb{E} \underline{VS^2}\langle \underline{G}\rangle^2 \langle \underline{T}\rangle+\frac{2}{N^2}\mathbb{E} \underline{VT^2}\langle \underline{G}\rangle^2 \langle \underline{S}\rangle-4\cal R^{(1)}\bigg)\,.
	\end{multline}
	By Lemmas \ref{prop_prec} and \ref{lem for s_1}, it is easy to see that the first three terms on the RHS of \eqref{5.7} are bounded by $O_{\prec}(\cal E)$. Further, note that
	\[
	\frac{m(z_1)-m(z_4)}{z_1-z_4} \prec 1+ \frac{|\sqrt{z_1^2-1}|+|\sqrt{z_4^2-1}|}{|z_1-z_4|}\prec 1+(u_1+u_4)\Big(\frac{k}{u_1}\wedge\frac{l}{u_4}\Big) \prec k+l
	\]
	uniformly for $z_1 \in \widehat{\gamma}_1$ and $z_4 \in \widehat{\gamma}_4$. This together with Lemma \ref{lem for s_1} and resolvent identity leads to
	\begin{multline}  \label{19082960}
	\ul{VG^2}=\frac{\ul{V}-\ul{G}}{(z_4-z_1)^2}-\frac{\ul{G^2}}{z_4-z_1} = \frac{m(z_4)-m(z_1)}{(z_4-z_1)^2}+\frac{\ul{V}-m(z_4)+m(z_1)-\ul{G}}{(z_1-z_4)^2}-\frac{\ul{G^2}}{z_4-z_1}\\
	\prec \frac{k+l}{|z_1-z_4|}+\frac{1}{|z_1-z_4|^2}\Big(\frac{k}{Nu_1}+\frac{l}{Nu_4}\Big)+\frac{1}{|z_1-z_4|u_1} \prec \frac{k+l}{|z_1-z_4|}
	\end{multline}
	uniformly for $z_1\in \widehat{\gamma}_1$ and $z_4 \in \widehat{\gamma}_4$. Hence by Lemmas \ref{prop_prec} and \ref{lem for s_1}, the fourth term on the RHS of \eqref{5.7} is bounded by
	\[
	\frac{1}{u_1u_4}\frac{1}{N^2}\frac{k+l}{|z_1-z_4|}\frac{k^2l}{N^3u_1u_2u_3}\prec \cal E\,.
	\]
	Similarly, the fifth and sixth terms on the RHS of \eqref{5.7} are also bounded by $O_{\prec}(\cal E)$. As in Lemma \ref{lem:5.2}, we can apply Lemma \ref{lem for s_1} to show that
	\[
	\frac{\cal R^{(1)}}{s_1s_4} \prec \frac{1}{u_1u_4} \cdot \frac{k^3l}{N^4u_1^2u_2u_3}\cdot \frac{1}{N^{1/2+1/6}} \prec \cal E\,.
	\]
	This completes the proof.
\end{proof}

\begin{lem} \label{lem:5.5}
	Let $L_2$ be as in (\ref{20070601}), we have
	\begin{equation*} 
	L_2/s_1 \prec \cal E\,.
	\end{equation*}
	\begin{proof}  By definition in (\ref{20070601}) and the differential rule \eqref{eqn:diff}, it is elementary to compute
		\begin{equation} \label{5.20}
		\begin{aligned}
		L_{2}/s_1=&\,\frac{1}{2s_1N}\sum_{i,j}\cal C_{3}(H_{ij})(1+\delta_{ij})^{-2}  \bb E\bigg( \la 6G_{ii}G_{jj}G_{ij}\rangle \langle \underline{S}\rangle\langle \underline{T}\rangle \langle \underline{V}\rangle+\langle 2G^3_{ij}\ra \langle \underline{S}\rangle \langle \underline{T}\rangle \langle \underline{V}\rangle\\
		&+\frac{4}{N}(G_{ii}G_{jj}+G_{ij}^2)\Big((S^2)_{ij} \langle \ul{T} \rangle \langle \ul{V} \rangle+(T^2)_{ij} \langle \ul{S} \rangle \langle \ul{V} \rangle+(V^2)_{ij} \langle \ul{S} \rangle \langle \ul{T} \rangle\Big)\\		
		&+\frac{2}{N}G_{ij}\Big((S^2)_{ii}S_{jj}+(S^2)_{jj}S_{ii}+2(S^2)_{ij}S_{ij}\Big)\langle \underline{T}\rangle \langle \underline{V}\rangle\\
		&+\frac{2}{N}G_{ij}\Big((T^2)_{ii}T_{jj}+(T^2)_{jj}T_{ii}+2(T^2)_{ij}T_{ij}\Big)\langle \underline{S}\rangle \langle \underline{V}\rangle\\
		&+\frac{2}{N}G_{ij}\Big((V^2)_{ii}V_{jj}+(V^2)_{jj}V_{ii}+2(V^2)_{ij}V_{ij}\Big)\langle \underline{S}\rangle \langle \underline{T}\rangle\\
		&+\frac{8}{N}G_{ij}\Big((S^2)_{ij}T^2_{ij}\langle \ul{V} \rangle +(S^2)_{ij}V^2_{ij}\langle \ul{T} \rangle +(T^2)_{ij}V^2_{ij}\langle \ul{S} \rangle \Big)\bigg)\,.
			\end{aligned}
		\end{equation} 
	The most dangerous terms on the RHS of the above equation appear in the second line, for example
		\begin{align*}
		L_{2,1}/s_1 &\deq \frac{2}{s_1N^2}\sum_{i,j} \cal C_3(H_{ij})(1+\delta_{ij})^{-2} \bb E G_{ii}G_{jj} (S^2)_{ij} \langle \ul{T} \rangle \langle \ul{V} \rangle\nonumber\\
		&=\frac{2}{s_1N^{7/2}}\sum_{i,j} a_{ij} \bb E G_{ii}G_{jj} (S^2)_{ij} \langle \ul{T} \rangle \langle \ul{V} \rangle\,,
		\end{align*}
		where in the second step we write $a_{ij}\deq N^{3/2} \cal C_{3}(H_{ij})(1+\delta_{ij})^{-2}$.  Observe that $\max_{i,j}|a_{ij}|=O(1)$. Applying $z_4V=HV-I$ and the cumulant expansion in Lemma \ref{lem:cumulant_expansion}, we obtain
		\begin{multline} \label{5.13}
		L_{2,1}/s_1=\frac{1}{-2s_1s_4N^{7/2}}\sum_{i,j}a_{ij}\bigg( \bb E G_{ii}G_{jj} (S^2)_{ij} \langle \ul{T} \rangle \langle \ul{V} \rangle^2 -\bb E\la\underline{V}\ra^2\bb E G_{ii}G_{jj} (S^2)_{ij} \langle \ul{T} \rangle \\+\frac{1}{N}\bb E G_{ii}G_{jj} (S^2)_{ij} \langle \ul{T} \rangle \langle \underline{V^2}\rangle+\frac{2}{N^2}\bb E\ul{VT^2} G_{ii}G_{jj} (S^2)_{ij}+\frac{2}{N^2}(G^2V)_{ii}G_{jj}(S^2)_{ij} \langle \ul{T}\rangle \\+\frac{2}{N^2}(G^2V)_{jj}G_{ii}(S^2)_{ij} \langle \ul{T}\rangle
		+ \frac{4}{N^2}(S^3V)_{ij}G_{ii}G_{jj} \langle \ul{T} \rangle -4\cal R^{(2,ij)}\bigg)\,.
		\end{multline}
		
		We can rewrite the first term on the RHS of the above equation as
	\begin{multline} \label{xianliang}
		\frac{-1}{2s_1s_4N^{7/2}}\sum_{i,j}a_{ij}\bb E \Big((G_{ii}-m(z_1))(G_{jj}-m(z_1))+m(z_1)(G_{ii}-m(z_1))\\+m(z_1)(G_{jj}-m(z_1))+m(z_1)^2\Big) (S^2)_{ij} \langle \ul{T} \rangle \langle \ul{V} \rangle^2\,.
	\end{multline}
	By Lemmas \ref{prop_prec} and \ref{lem for s_1}, the first term in \eqref{xianliang} is stochastically dominated by
	\[
	\frac{1}{s_1s_4N^{7/2}} \cdot N^2 \cdot \sqrt{\frac{k}{u_1N}} \cdot \sqrt{\frac{k}{u_1N}} \cdot \frac{1}{u_2} \cdot \frac{l^3}{u_3u^2_4N^3} \asymp \frac{kl^3}{u_1^2u_2u_3u^3_4N^{11/2}}\prec \frac{kl}{u_1^2u_2u_3u^2_4N^{4+\zeta}}\leq \cal E\,.
	\]
	In order to estimate the others terms in \eqref{xianliang}, we first define the vector $\b v^j=(v^j_1, \ldots, v_N^j)^{\top}$ by $v_i^j\deq N^{-1/2}a_{ij}$, for convenience. Note that  we have $\sup_j\|\b v^j\|_2=O(1)$.  Then by Lemma \ref{lem for s_1} we have
	\begin{equation} \label{5.222}
	\sum_i a_{ij} S_{ij}=N^{1/2} \langle \overline{\b v}^j, S \b e_j\rangle  \prec N^{1/2}
	\end{equation}
and
	\begin{equation} \label{5.23}
\sum_i a_{ij} (S^2)_{ij}=N^{1/2} \langle \overline{\b v}^j, (S^2) \b e_j\rangle  \prec \frac{N^{1/2}}{u_2^{3/2}}
	\end{equation}
as well as
	\begin{equation} \label{5.24}
	\sum_i a_{ij} (S^3)_{ij}=\sum_{i,k}a_{ij}(S^2)_{ik}S_{kj}=\sum_{k:k\ne j}N^{1/2} \langle \overline{\b v}^j, (S^2) \b e_k\rangle S_{kj}+N^{1/2} \langle \overline{\b v}^j, (S^2) \b e_j\rangle S_{jj}\prec \frac{Nk^{1/2}}{u_2^2}\,.
	\end{equation}
	Using \eqref{5.23}, the last term in \eqref{xianliang} is stochastically dominated by
	\[
	\frac{1}{s_1s_4N^{7/2}}\cdot N\cdot \frac{N^{1/2}}{u_2^{3/2}}\cdot \frac{l^3}{u_3u_4N^3} \asymp \frac{l^3}{u_1u_2^{3/2}u_3u_4^3N^5}\prec \frac{l}{u_1u_2^{3/2}u_3u_4^2N^{4+\zeta}} \leq \cal E\,.
	\]
	The second and third term in \eqref{xianliang} can be estimated in the same way. This implies
	\[
	\eqref{xianliang} \prec \cal E\,.
	\]
	This concludes the estimate of the first term  on the RHS of \eqref{5.13}. 
   In the same way, the second and third terms on the RHS of \eqref{5.13} can also bounded by $O_{\prec}(\cal E)$. 
   
   The estimates for the fourth to seventh terms on RHS of \eqref{5.13} follows a similar fashion. As in \eqref{19082960}, we can show that
   \begin{equation} \label{db}
   \ul{VT^2} \prec \frac{l}{|z_3-z_4|}\,.
    \end{equation}
  By Lemma \ref{lem for s_1}, \eqref{5.23} and \eqref{db}, together with the identity
  \begin{multline*}
   \ul{VT^2} G_{ii}G_{jj} (S^2)_{ij}
  \\ =\ul{VT^2}
   \Big((G_{ii}-m(z_1))(G_{jj}-m(z_1))+m(z_1)(G_{ii}-m(z_1))+m(z_1)(G_{jj}-m(z_1))+m(z_1)^2\Big) (S^2)_{ij}\,,
    \end{multline*}
   we can bound the fourth term on RHS of \eqref{5.13} by $O_{\prec}(\cal E)$. The fifth term on RHS of \eqref{5.23} can be rewritten into
   \begin{multline*} 
   \frac{-1}{s_1s_4N^{11/2}}\sum_{i,j}a_{ij}\bb E\bigg[ \bigg(\frac{(G^2)_{ii}}{z_1-z_4}-\frac{G_{ii}-m(z_1)+m(z_4)-V_{ii}}{(z_1-z_4)^2}+\frac{m(z_1)-m(z_4)}{(z_1-z_4)^2} \bigg)\\
   \times \big(m(z_1)+G_{jj}-m(z_1)\big)(S^2)_{ij} \langle\ul{T} \rangle \bigg]
   \end{multline*}
    which is bounded by $O_{\prec}(\cal E)$ using Lemma \ref{lem for s_1} and \eqref{5.23}. The sixth term on RHS of  \eqref{5.13} equals  the fifth term. Using the identity
    \begin{multline*}
    (S^3V)_{ij}G_{ii}G_{jj} \langle \ul{T} \rangle=\bigg(\frac{(S^3)_{ij}}{z_2-z_4}-\frac{(S^2)_{ij}}{(z_2-z_4)^2}+\frac{S_{ij}-V_{ij}}{(z_2-z_4)^3}\bigg)\\
    \times\Big((G_{ii}-m(z_1))(G_{jj}-m(z_1))+m(z_1)(G_{ii}-m(z_1))+m(z_1)(G_{jj}-m(z_1))+m(z_1)^2\Big) \langle \ul{T} \rangle 
     \end{multline*}
     together with  Lemma \ref{lem for s_1}, \eqref{5.222} -- \eqref{5.24}, as well as the bound
     \[
     \frac{1}{|z_2-z_4|} =O(k^2 \wedge l^2)\,,
     \]
      the seventh term on RHS of \eqref{5.13} can also be bounded by $O_{\prec}(\cal E)$. Similar to the proof of Lemma \ref{lem:5.2}, we can apply Lemma \ref{lem for s_1} to show that
		\[
		\frac{1}{s_1s_4N^{7/2}}\sum_{i,j}a_{ij}\cal R^{(2,ij)} \prec \cal E\,.
		\]
		Thus we conclude the estimate of $L_{2,1}/s_1$. 
		
		 Similarly we can show that
		\[
		\frac{2}{s_1N^2}\sum_{i,j} \cal C_3(H_{ij})(1+\delta_{ij})^{-2} \bb E G_{ii}G_{jj}\big( (T^2)_{ij} \langle \ul{S} \rangle \langle \ul{V} \rangle+(V^2)_{ij} \langle \ul{S} \rangle \langle \ul{T} \rangle\big) \prec \cal E\,.
		\] 
		The estimates for other terms on the RHS of \eqref{5.20} are easier, namely one only needs to apply Lemma \ref{lem for s_1} and no further cumulant expansion is needed. For example, the first term on the RHS of \eqref{5.20} can be rewritten into
		\begin{multline*}
		\frac{3}{s_1N}\sum_{i,j}\cal C_{3}(H_{ij})(1+\delta_{ij})^{-2}  \bb E\Big((G_{ii}-m(z_1))(G_{jj}-m(z_1))+m(z_1)(G_{ii}-m(z_1))\\+m(z_1)(G_{jj}-m(z_1))+m(z_1)^2\Big) G_{ij} \langle \ul{S} \rangle\langle \ul{T} \rangle \langle \ul{V} \rangle^2\,.
		\end{multline*}
		Similar as in \eqref{xianliang}, the above can be easily bounded by $O_{\prec}(\cal E)$ using Lemma \ref{lem for s_1}. The second term on the RHS of \eqref{5.20} is stochastically dominated by
		\[
		\frac{1}{u_1N}\cdot N^{-3/2}\Big(N+\sum_{i,j: i\ne j} \bb E |G^3_{ij}|\Big) \cdot \frac{kl^2}{u_2u_3u_4N^3} \prec\frac{kl^2}{u_1u_2u_3u_4N^{9/2}}+ \frac{k^{5/2}l^2}{u_1^{5/2}u_2u_3u_4N^5}\prec \cal E\,.
		\]   
		One readily checks that all terms on the RHS of \eqref{5.20} are bouned by $O_{\prec}(\cal E)$. This completes the proof.
	\end{proof}
\end{lem}

In the sequel, we estimates the terms $L_3/s_1$ and $K/s_1$ in (\ref{5.5}), which are not negligible. The results are stated in Lemmas \ref{lem.19101701} and \ref{lem.19101702}, followed by the proofs. 

\begin{lem} \label{lem.19101701}
	Let $L_3$ be as in (\ref{20070601}),  we have
	\begin{multline} \label{20070610}
	L_3/s_1=-\frac{2}{s_1N^2}\sum_{i,j}\cal C_4(H_{ij}) \bb E G_{ii}G_{jj}\Big((S^2)_{ii}S_{jj} \langle\ul{T}\rangle \langle \ul{V} \rangle
	+(T^2)_{ii}T_{jj} \langle\ul{S}\rangle \langle \ul{V} \rangle+(V^2)_{ii}V_{jj} \langle\ul{S}\rangle \langle \ul{T} \rangle\Big)+O_{\prec}(\cal E)\,.
	\end{multline}
\end{lem}
\begin{proof}
 By definition 
\begin{equation} \label{5.26}
L_3/s_1=\frac{1}{6s_1N}\sum_{i,j}\cal C_{4}(H_{ij}) \bb E\partial^3_{ij} \big(G_{ji}\big\la\langle \underline{S}\rangle \langle \underline{T}\rangle \langle \underline{V}\rangle\big\ra\big)\,.
\end{equation}
Using \eqref{eqn:diff}, we easily see that the leading terms on the RHS of \eqref{20070610} come from applying one $\partial_{ij}$ on $G_{ij}$,  and then applying the remaining two $\partial_{ij}$ either both on $\langle \ul{S} \rangle $, $\langle \ul{T} \rangle $, or $\langle \ul{V} \rangle $. For example, the first term on the RHS of \eqref{20070610} is contained in
	\[
	\frac{1}{6}\frac{1}{N}\sum_{i,j}\cal C_{4}(H_{ij}) \bb E\Big(3(\partial_{ij} G_{ji})(\partial^2_{ij}\langle \underline{S}\rangle) \langle \underline{T}\rangle \langle \underline{V}\rangle\Big)\,.
	\]
    Aside from the explicit terms we have on RHS of \eqref{20070610}, the remaining terms on the RHS of \eqref{5.26} are the error terms.  Among them, the most difficult term comes from applying $\partial_{ij}^3$ all on $G_{ij}$, which may generate four diagonal entries of the Green function. More specifically, the following term is the most dangerous on for the estimation
	\[
	L_{3,1}/s_1\deq	-\frac{1}{s_1N} \sum_{i,j} \cal C_4(H_{ij}) \bb E \langle G_{ii}^2G_{jj}^2 \rangle \langle\ul{S} \rangle \langle\ul{T} \rangle \langle \ul{V} \rangle\,.
	\]
	To estimate this term, we need one additional expansion. By writing $z_1G_{ii}=(HG)_{ii}-1$ and using the cumulant expansion in Lemma \ref{lem:cumulant_expansion}, we have
	\begin{multline*}
	L_{3,1}/s_1=\frac{1}{(4z_1+\bb E \ul{G})s_1N}\sum_{i,j}\cal C_4(H_{ij})\bb E \Big(\big\langle G_{ii}^2G_{jj}^2\langle\ul{G}\rangle+3N^{-1}(G^2)_{ii}G_{jj}^2+5N^{-1}G_{ij}(G^2)_{ij} \big\rangle \langle\ul{S} \rangle \langle\ul{T} \rangle \langle \ul{V} \rangle\\
	+2N^{-2}(GS^2)_{ii}G_{ii}G^2_{jj} \langle \ul{T}\rangle \langle \ul{V}\rangle+2N^{-2}(GT^2)_{ii}G_{ii}G^2_{jj} \langle \ul{S}\rangle \langle \ul{V}\rangle+2N^{-2}(GV^2)_{ii}G_{ii}G^2_{jj} \langle \ul{S}\rangle \langle \ul{T}\rangle-4\cal R^{(3,ij)}\Big)\,.
	\end{multline*}
Note that Lemma \ref{Tlemh} implies $\cal C_4(H_{ij})=O(N^{-2})$, uniformly in $i,j$. By $1/|4z_1+\bb E \ul{G}|\approx 1/|4z_1+m(z_1)|=O(1)$ and Lemma \ref{lem for s_1}, one readily follows the approach in Lemma \ref{lem:5.2} and checks that
	\[
	\frac{-4}{(4z_1+\bb E \ul{G})s_1N}\sum_{i,j}\cal C_4(H_{ij}) \cal R^{(3,ij)} \prec\cal E\,.
	\] 
	Again by $1/|4z_1+\bb E \ul{G}|=O(1)$ and Lemma \ref{lem for s_1}, other terms in $L_{3,1}/s_1$ can also be shown to satisfy the same bound. This implies
	\[
	L_{3,1}/s_1 \prec \cal E\,.
	\] 
	Other error terms on the RHS of \eqref{5.26} can be estimated directly using Lemma \ref{lem for s_1}, as in the proofs of Lemmas \ref{lem:5.4} and \ref{lem:5.5}. For example, we will have terms of the types
	\begin{equation} \label{5.27}
	-\frac{1}{s_1N} \sum_{i,j} \cal C_4(H_{ij}) \bb E \langle G_{ii}G_{jj}G_{ij}^2 \rangle \langle\ul{S} \rangle \langle\ul{T} \rangle \langle \ul{V} \rangle
	\end{equation}
	and
	\begin{equation} \label{5.28}
	-\frac{1}{s_1N^2} \sum_{i,j} \cal C_4(H_{ij}) \bb E  G_{ii}G_{jj} (S^2)_{ij}S_{ij} \langle\ul{T} \rangle \langle \ul{V} \rangle
	\end{equation}
	as well as 
	\begin{equation} \label{5.29}
	-\frac{1}{s_1N^3} \sum_{i,j} \cal C_4(H_{ij}) \bb E  G_{ii}G_{jj} (S^2)_{ij}S_{ij} (T^2)_{ij}T_{ij} \langle \ul{V} \rangle\,.
	\end{equation}
Applying Lemma \ref{lem for s_1}, we can easily get
\begin{equation*}
\eqref{5.27} \prec \frac{1}{u_1N} \cdot N^{-2}\Big(N+\sum_{i,j:i \ne j} \bb E |G_{ij}^2| \Big)\cdot  \frac{kl^2}{u_2u_3u_4N^3}\prec\frac{kl^2}{u_1u_2u_3u_4N^5}+\frac{k^2l^2}{u_1^2u_2u_3u_4N^5} \prec \cal E
\end{equation*}
and
\[
\eqref{5.28} \prec \frac{1}{u_1N^2} \cdot N^{-2} \Big(N+\sum_{i,j:i\ne j}\bb E|S_{ij}|\Big)\cdot \frac{1}{u_2^{3/2}}\cdot \frac{l^2}{u_3u_4N^2}\prec \frac{l^2}{u_1u_2^{3/2}u_3u_4N^5}+\frac{k^{1/2}l^2}{u_1u_2^2u_3u_4N^{9/2}} \prec \cal E
\]
as well as
\[
\eqref{5.29} \prec \frac{1}{u_1N^3} \cdot N^{-2} \Big(N+\sum_{i,j:i\ne j} \bb E |S_{ij}|\Big)\frac{1}{u_2^{3/2}}\cdot \frac{1}{u_3^{3/2}} \cdot \frac{l}{Nu_4} \prec \cal E\,.
\]
Using the above argument, we can bound all the error terms on the RHS of \eqref{5.26} by $O_{\prec}(\cal E)$. This completes the proof.
\end{proof}
\begin{lem} \label{lem.19101702}
	Let $K$ be as in (\ref{20070605}), we have
	\[
	K/s_1=-\frac{\sigma^2-2}{4s_1N^3}\sum_i \bb E\Big( G_{ii} (S^2)_{ii}\la \ul{T} \ra \la \ul{V} \ra +\bb E G_{ii} (T^2)_{ii}\la \ul{S} \ra \la \ul{V} \ra +\bb E G_{ii} (V^2)_{ii}\la \ul{S} \ra \la \ul{T} \ra \Big)+O_{\prec} (\cal E)\,.
	\]
\end{lem}
\begin{proof}
	The estimate can be done similarly to those of Lemmas \ref{lem:5.5} and \ref{lem.19101701}, and thus we only give a sketch. From the definition and \eqref{eqn:diff}, it is easy to derive
	\begin{multline} \label{eqn: K}
	K/s_1=-\frac{\sigma^2-2}{4s_1N^2} \sum_i \Big(\bb E \la G_{ii}^2\ra \la \ul{S} \ra \la \ul{T} \ra \la \ul{V} \ra \\ +\frac{1}{N}\bb E G_{ii} (S^2)_{ii}\la \ul{T} \ra \la \ul{V} \ra +\frac{1}{N}\bb E G_{ii} (T^2)_{ii}\la \ul{S} \ra \la \ul{V} \ra +\frac{1}{N}\bb E G_{ii} (V^2)_{ii}\la \ul{S} \ra \la \ul{T} \ra \Big)\,.
	\end{multline}
	By Theorem \ref{thm. local semicircle law} and Lemma \ref{lem for s_1}, the first term on the RHS of \eqref{eqn: K} can be bounded by
	\[
	O_{\prec}\Big(	\frac{1}{u_1N^2} \cdot N \cdot \sqrt{\frac{k}{Nu_1}} \frac{kl^2}{N^3u_2u_3u_4}\Big) =O_{\prec}(\cal E)\,.
	\]
	This completes the proof.
\end{proof}

\subsection{Proof of Proposition \ref{prop: centered}} \label{sec: 5.3}
Now we insert Lemmas \ref{lem:5.2} -- \ref{lem.19101702} to \eqref{5.5}, and we get after rearranging the terms 
\begin{multline}  \label{5.16}
\mathbb{E}\langle \underline{G}\rangle \langle \underline{S}\rangle \langle \underline{T}\rangle \langle \underline{V}\rangle=  -\frac{1}{2s_1N^2}\mathbb{E} \Big(\underline{GS^2}+4\sum_{i,j}\cal C_4(H_{ij})  G_{ii}G_{jj}(S^2)_{ii}S_{jj}+\frac{\sigma^2-2}{2N}\sum_i G_{ii}(S^2)_{ii}\Big) \langle\ul{T}\rangle \langle \ul{V} \rangle\\
-\frac{1}{2s_1N^2}\mathbb{E} \Big(\underline{GT^2}+4\sum_{i,j}\cal C_4(H_{ij})   G_{ii}G_{jj}(T^2)_{ii}T_{jj}+\frac{\sigma^2-2}{2N}\sum_i G_{ii}(T^2)_{ii}\Big) \langle\ul{S}\rangle \langle \ul{V} \rangle\\
-\frac{1}{2s_1N^2}\mathbb{E} \Big(\underline{GV^2}+4\sum_{i,j}\cal C_4(H_{ij})  G_{ii}G_{jj}(V^2)_{ii}V_{jj}+\frac{\sigma^2-2}{2N}\sum_i G_{ii}(V^2)_{ii}\Big) \langle\ul{S}\rangle \langle \ul{T} \rangle+ O_{\prec}(\cal E)\,.
\end{multline}

The terms on the RHS of \eqref{5.16} can be further computed using the following lemma. The proof is again done by cumulant expansion and Lemma \ref{lem for s_1}. We omit the details.
\begin{lem} \label{lem.200701}
	We have
	\begin{multline*}
	-\frac{1}{2s_1N^2}\mathbb{E} \Big(\underline{GS^2}+4\sum_{i,j} \cal C_4(H_{ij})  G_{ii}G_{jj}(S^2)_{ii}S_{jj}+\frac{\sigma^2-2}{2N}\sum_i G_{ii}(S^2)_{ii}\Big) \langle\ul{T}\rangle \langle \ul{V} \rangle\\
	=\frac{1}{4s_1s_3N^4}\mathbb{E} \Big(\underline{GS^2}+4\sum_{i,j} \cal C_4(H_{ij})  G_{ii}G_{jj}(S^2)_{ii}S_{jj}+\frac{\sigma^2-2}{2N}\sum_i G_{ii}(S^2)_{ii}\Big)\\
	\times \Big(\underline{TV^2}+4\sum_{i,j} \cal C_4(H_{ij}) T_{ii}T_{jj}(V^2)_{ii}V_{jj}+\frac{\sigma^2-2}{2N}\sum_i T_{ii}(V^2)_{ii}\Big)+O_{\prec}(\cal E)
	\\=f(z_1,z_2)f(z_3,z_4)+O_{\prec}(\cal E)\,,
	\end{multline*}
	where $f(\cdot,\cdot)$ is defined as in \eqref{5.18}.
	
\end{lem}
Note that the first three leading terms on the RHS of (\ref{5.16}) have the same form. Applying Lemma \ref{lem.200701} to all three terms in (\ref{5.16}) with appropriate permutations of $z_1,z_2,z_3,z_4$,  we obtain
\begin{equation*}
\mathbb{E}\langle \underline{G}\rangle \langle \underline{S}\rangle \langle \underline{T}\rangle \langle \underline{V}\rangle=f(z_1,z_2)f(z_3,z_4)+f(z_1,z_3)f(z_2,z_4)+f(z_1,z_4)f(z_2,z_3)+O_{\prec}(\cal E)\,.
\end{equation*}
By repeating the above proof, one also gets that
\begin{equation} \label{eqn: two point}
\mathbb{E}\langle \underline{G}\rangle \langle \underline{S}\rangle=f(z_1,z_3)+O_{\prec}(kt_{1,2}N^{-2-\zeta})
\end{equation}
and
\begin{equation*}
 \bb E \langle \underline{T}\rangle \langle \underline{V}\rangle=f(z_3,z_4)+O_{\prec}(l t_{3,4}N^{-2-\zeta})\,.
\end{equation*}
The above three equations conclude the proof of \eqref{19052450}.

\subsection{The explicit shift} \label{sec 5.4}
In this section we replace the centered random  variables $\langle \ul{G}(z_i) \rangle $ in \eqref{19052450} by the one with an explicit shift $\ul{G}(z_i)-m(z_i)=\langle \ul{G}(z_i) \rangle+\bb E \ul{G}(z_i)-m(z_i)$, which is needed for the proof of Proposition \ref{pro.19052401}  according to  the definition of $m_N^\Delta$ in (\ref{def. t hat}). The following result is stated for $z_1$ only for convenience, but it still holds if we replace  $(z_1,\widehat{\gamma}_1)$ by any other $(z_i,\widehat{\gamma}_i), i=2,3,4$.
\begin{lem} \label{lem: expectstion}
	Recall $g(\cdot)$ from \eqref{eqn: g}. 	We have
	\[
	\bb E \ul{G}-m(z_1)=g(z_1)+O_{\prec}\Big( \frac{k}{N^{3/2}u_1^3}\Big)\,.
	\]
	uniformly for $z_1 \in \widehat{\gamma}_{1}$.
\end{lem}
\begin{proof}
	Using $z_1\bb E\ul{G}=\bb E \ul{HG}-1$ and the cumulant expansion formula, we have
	\begin{equation} \label{5.21}
	(\bb E \ul{G})^2+4z_1\bb E \ul{G}+4+\frac{1}{N}\bb E \ul{G^2}+\bb E \la \ul{G}\ra^2-4L_2^{(4)}-4L_3^{(4)}-4K^{(4)}-4\cal R^{(4)}=0\,,
	\end{equation}
	where \[
	L_n^{(4)} \deq \frac{1}{n!}\frac{1}{N}\sum_{i,j}\cal C_{n+1}(H_{ij}) \bb E\partial^n_{ij} G_{ji}\,, \quad K^{(4)}=-\frac{\sigma^2-2}{4N^2}\sum_i\bb E G^2_{ii}
	\]
	and $\cal R^{(4)}$ is the remainder term. By Lemma \ref{lem for s_1} we can easily check
	\begin{equation} \label{5.22}
	-4L_3^{(4)}=\frac{4c_4}{N} (m(z_1))^4+O_{\prec}\Big(\frac{k}{N^2u_1}\Big)\,, 
	\end{equation}
	\begin{equation} -4K^{(4)}=\frac{\sigma^2-2}{N}(m(z_1))^2+O_{\prec}\Big(\frac{k^{1/2}}{N^{3/2}u^{1/2}}\Big)\,,
	\end{equation}
	and
	\begin{equation} \label{5.31}
	-4L_2^{(4)}-4\cal R^{(4)}\prec \Big(\frac{k}{Nu_1}\Big)^{3/2}\frac{1}{\sqrt{N}}+\frac{1}{N^{3/2}}\,.
	\end{equation}
	Similarly, by cumulant expansion and Lemma \ref{lem for s_1}, we have
	\begin{equation} \label{5.32}
	\bb E  \langle \underline{G} \rangle^2=\ \frac{1}{-4z-2\bb E \ul{G}} \Big(\bb E  \langle \underline{G} \rangle^{3}
	+\frac{1}{N}\bb E \langle \underline{G} \rangle\langle \underline{G^2} \rangle+\frac{2}{N^2}\bb E \underline{G^3}
	-4\cal R^{(5)}\Big)\prec \frac{k}{N^{3/2}u_1^2}
	\end{equation}
	and
	\begin{equation} \label{5.25}
	\bb E \ul{G^2}=\frac{1}{-4z-2\bb E \ul{G}}\Big(4\bb E \ul{G}+\frac{1}{N}\bb E \ul{G^3}+2\bb E  \la \ul{G}\ra \la \ul{G^2}\ra -4\cal R^{(6)}\Big)=m'(z_1)+O_{\prec}\Big(\frac{k}{N^{1/2}u_1^{2}}\Big)\,.
	\end{equation}
 In \eqref{5.31} -- \eqref{5.25} above, we estimate the remainder terms $\cal R^{(4)}$, $\cal R^{(5)}$, $\cal R^{(6)}$ using the method presented in the proof of Lemma \ref{lem:5.2}; see also the discussion in Remark \ref{rmk:5.4}. \black By \eqref{5.22} -- \eqref{5.25}, we can rewrite \eqref{5.21} as
	\begin{equation*}
	(\bb E \ul{G})^2+4z_1\bb E \ul{G}+4+\frac{1}{N}\Big( m'(z_1)+4\,c_4 (m(z_1))^4+(\sigma^2-2)(m(z_1))^2\Big)=O_{\prec}\Big( \frac{k}{N^{3/2}u_1^2}\Big)\,,
	\end{equation*}
	which implies the desired result.
\end{proof}

Combining Proposition \ref{prop: centered} with Lemma \ref{lem: expectstion}, we can now prove Proposition \ref{pro.19052401} in the sequel. 
\begin{proof}[Proof of Proposition \ref{pro.19052401} ]
Equipped with Lemma \ref{lem: expectstion}, we can follow a similar computation as in Section \ref{sec:estimate} to show that
\begin{equation*}
\big(\bb E \ul{G}-m(z_1)\big)\mathbb{E} \langle \underline{S}\rangle \langle \underline{T}\rangle \langle \underline{V}\rangle =O_{\prec}(\cal E)\,,
\end{equation*}
which implies
\begin{multline*}
\bb E\big(m_N^\Delta(z_1)m_N^\Delta(z_2)m_N^\Delta(z_3)m_N^\Delta(z_4)\big)-\mathbb{E}\langle \underline{G}\rangle \langle \underline{S}\rangle \langle \underline{T}\rangle \langle \underline{V}\rangle\\
=\frac{1}{2}\sum_{\{a,b,c,d\}=\{1,2,3,4\}}(\bb E \ul{G}(z_a)-m(z_a))(\bb E \ul{G}(z_b)-m(z_b))\bb E  \langle \underline{G}(z_c)\rangle \langle \underline{G}(z_d)\rangle\\
+(\bb E \ul{G}-m(z_1))(\bb E \ul{S}-m(z_2))(\bb E \ul{T}-m(z_3))(\bb E \ul{V}-m(z_4))+O_{\prec}(\cal E)\,,
\end{multline*}
and
\begin{multline*}
\bb E\big(m_N^\Delta(z_1)m_N^\Delta(z_2)\big)\bb E \big(m_N^\Delta(z_3)m_N^\Delta(z_4)\big)-\mathbb{E}\langle \underline{G}\rangle \langle \underline{S}\rangle\bb E \langle \underline{T}\rangle \langle \underline{V}\rangle\\
=(\bb E \ul{G}-m(z_1))(\bb E \ul{S}-m(z_2))\bb E  \langle \underline{T}\rangle \langle \underline{V}\rangle+(\bb E \ul{T}-m(z_3))(\bb E \ul{V}-m(z_4))\bb E  \langle \underline{G}\rangle \langle \underline{S}\rangle
\\
+(\bb E \ul{G}-m(z_1))(\bb E \ul{S}-m(z_2))(\bb E \ul{T}-m(z_3))(\bb E \ul{V}-m(z_4))+O_{\prec}(\cal E)\,.
\end{multline*}
Thus we have
\begin{multline*}
\cov \big(m_N^\Delta(z_1)m_N^\Delta(z_2),m_N^\Delta(z_3)m_N^\Delta(z_4)\big)-\cov \big(\langle \underline{G}\rangle \langle \underline{S}\rangle, \langle \underline{T}\rangle \langle \underline{V}\rangle\big)\\
=(\bb E \ul{G}-m(z_1))(\bb E \ul{T}-m(z_3))\bb E  \langle \underline{S}\rangle \langle \underline{V}\rangle+(\bb E \ul{G}-m(z_1))(\bb E \ul{V}-m(z_4))\bb E  \langle \underline{S}\rangle \langle \underline{T}\rangle\\
+(\bb E \ul{S}-m(z_2))(\bb E \ul{V}-m(z_4))\bb E  \langle \underline{G}\rangle \langle \underline{T}\rangle+(\bb E \ul{S}-m(z_2))(\bb E \ul{T}-m(z_3))\bb E  \langle \underline{G}\rangle \langle \underline{V}\rangle+O_{\prec}(\cal E)\,.
\end{multline*}
By computing $\bb E \la \ul{G}(z_i) \ra \la \ul{G}(z_j) \ra $ as in \eqref{eqn: two point} and applying Lemma \ref{lem: expectstion}, we have
\begin{multline*}
\cov \big(m_N^\Delta(z_1)m_N^\Delta(z_2),m_N^\Delta(z_3)m_N^\Delta(z_4)\big)-\cov \big(\langle \underline{G}\rangle \langle \underline{S}\rangle, \langle \underline{T}\rangle \langle \underline{V}\rangle\big)\\
=g(z_1)g(z_3)f(z_2,z_4)+g(z_1)g(z_4)f(z_2,z_3)\\
+g(z_2)g(z_4)f(z_1,z_3)+g(z_2)g(z_3)f(z_1,z_4)+O_{\prec}(\cal E)\,.
\end{multline*}
Using the above relation and Proposition \ref{prop: centered}, we conclude the proof of \eqref{19052411}. The proof of (\ref{19052410}) is similar and simpler. We thus omit the details. Hence, we conclude the proof of Proposition \ref{pro.19052401}.
\end{proof}

\section{Further discussion} \label{section 6}

In this section, we make some further remarks on the CvM statistics.

\vspace{2ex}

\noindent$\bullet$ {\it Shortcoming of the CvM statistics}

From application point of view, although the original CvM statistic $\mathcal{A}_N$
is a robust statistic, it has its own shortcoming. For instance, the statistic $\mathcal{A}_N$ is not sensitive to the strength of a low rank deformation of Wigner matrix. By Cauchy interlacing property, a rank one deformation can only cause a change of order $\frac{1}{N}$ for $F_N(t)$. This fact does not depend on the strength of the deformation of the Wigner matrix.   Hence, the power of the statistic $\mathcal{A}_N$ will not be significantly good even if the rank one deformation is very large, if we use $\mathcal{A}_N$ to test the existence of the deformation, say. In other words, the statistic $\mathcal{A}_N$ is not sensitive to the possible outlier of the spectrum. The same problem exists for our MCvM $\mathcal{A}_{N,w}$ since we use $\bar{\lambda}_i$ instead of $\lambda_i$
in the definition of $F_{N,\omega}$ in (\ref{190830110}).  However, from the proof of the main result Theorem \ref{thm. main result}, it is clear that  we indeed have the same convergence as  Theorem \ref{thm. main result} for the following partial sum 
\begin{align*}
\widetilde{\mathcal{A}}_{N,\omega}= &\frac{2}{\pi^2}\sum_{k=1}^{n_\omega} \frac{1}{k^2}  \big(t_k(H)\big)^2- \frac{2}{\pi^2}\sum_{k=1}^{n_\omega} \frac{1}{k(k+2)}t_k(H)t_{k+2}(H)+\frac{1}{\pi^2} \big(t_1(H)\big)^2. 
\end{align*}
as long as $n_\omega\leq N^{\frac13-\varepsilon}$, where $\bar{H}$ is replaced by $H$.  From the application point of view, if we use the statistic $\widetilde{\mathcal{A}}_{N,\omega}$, it is expected to be sensitive to the strength of the low rank deformation since now we have $H$ instead of $\bar{H}$. In this sense, the partial sum statistic such as $\widetilde{\mathcal{A}}_{N,\omega}$ has its own advantage in contrast to the original $\mathcal{A}_N$.

\vspace{2ex}

\noindent $\bullet$ {\it On expectation of $\mathcal{A}_N$}

According to  the simulation result in Figures \ref{fig2} and \ref{fig3}, we can formulate the following conjecture

\begin{con} Under the assumption of Theorem \ref{thm. main result}, we conjecture that the following holds
\begin{multline*}
N^2(\mathcal{A}_{N}-\mathbb{E}\mathcal{A}_{N})\overset{d}{\longrightarrow}  \frac{1}{\beta\pi^2}\sum_{k=1}^\infty\bigg( \frac{1}{ k}(Z_k^2-1)-\frac{1}{\sqrt{k(k+2)}}Z_kZ_{k+2}\bigg)\\
+\frac{2-\beta}{\sqrt{\beta}\pi^2}\sum_{k=1}^\infty\Big(\frac{k+2}{4k^{3/2}(k+1)}Z_{2k}-\frac{1}{4k\sqrt{k+1}}Z_{2k+2}\Big)+a_\beta\,,
\end{multline*}
where $a_\beta$ is defined in Theorem \ref{thm. main result}. 
\end{con}
In addition to the proof of the above conjecture (in case it is true), for application purpose, it is also necessary to identify $N^2\mathbb{E}\mathcal{A}_{N}$ up to the constant order, since the RHS  is an order 1 random variable. According to the definition of $\mathcal{A}_N$ in (\ref{190529100}) and the rigidity in (\ref{19052221}), one can cut the integral in  (\ref{190529100}) to the domain $t\in [-1+N^{-\varepsilon}, 1-N^{-\varepsilon}]$. Hence, roughly speaking, in order to identify $N^2\mathbb{E}\mathcal{A}_{N}$ up to the constant order, it would be enough to 
estimate $\mathbb{E}(F_N(t)-F(t))^2$ for  $t\in [-1+N^{-\varepsilon}, 1-N^{-\varepsilon}]$, up to the order of $\frac{1}{N^2}$. Write 
\begin{align*}
\mathbb{E}(F_N(t)-F(t))^2= \text{Var}(F_N(t))^2+(\mathbb{E}F_N(t)-F(t))^2.
\end{align*}
Even in the case of GUE/GOE,  only the first order term of $\text{Var}(F_N(t))^2$ is known ( see \cite{Gu05, OR10}), which is of order $\frac{\log N}{N^2}$ in the bulk. The subleading order is not available so far. Starting from the representation in (\ref{19092801}), for  $N^2\mathbb{E}\mathcal{A}_{N}$, one can also turn to identify $N^2\mathbb{E} (\lambda_i-\mu_i)^2$ up to the constant order. However, again, from the reference such as \cite{Gu05, OR10}, only the leading term of order $\log N$ is precisely available.  It is worth mentioning that in the recent work \cite{LS18}, a precise estimate on the constant order of $N^2(\text{Var}(\lambda_i(H))-\text{Var}(\lambda_i(\text{GOE})))$ is obtained in the bulk, see Theorem 1.4 therein. Here $\lambda_i(H)$ means the $i$-th largest eigenvalue of  a general Wigner matrix $H$ and $\lambda_i(\text{GOE})$ means the counterpart for a GOE. Unfortunately, the constant order  term of  $N^2\text{Var}(\lambda_i(\text{GOE}))$ itself is not available so far.

\vspace{2ex}

\noindent $\bullet$ {\it Sample covariance matrices counterpart}

 We remark here that all the discussion and result of the current work for Wigner matrix can be adapted to the sample covariance matrices. Considering the importance of covariance matrices in statistics theory, the CvM statistic and the mesoscopic approximations are potentially useful in many hypothesis testing problems. Especially, as we mentioned earlier, such a statistic has been used in \cite{WP} for testing the structure of the covariance matrices. Specifically, for an analogue of the result in this paper for covariance matrix, one needs to consider the expansion of the spectral distribution  
using the basis of the  shifted Chebyshev polynomials of the first kind. For instance, we refer to \cite{DE06} for the fact that the shifted Chebyshev polynomials diagonalize the covariance structure of the linear spectral statistics of the sample covariance matrices.   The detailed derivation for the counterpart of the current result for sample covariance matrices and the discussion for its applications will be considered in a future work.

\vspace{2ex}
{\bf Acknowledgement}: {We would like to thank Jiang Hu for discussion and simulation. We would also like to thank Yan Fyodorov, Gaultier Lambert, Dong Wang and Lun Zhang for reference and helpful comments.  Finally, we  are grateful to  the anonymous referee for  useful remarks and comments. }

\appendix

\section{A toy model: CUE} \label{App:A}

In this appendix, we consider a Cram\'{e}r-von Mises type statistic for the Circular Unitary Ensemble (CUE). Let $U$ be a $N$-dimensional CUE, which is a Haar distributed unitary matrix.  And we denote its (unordered) eigenvalues by $\mathrm{e}^{\mathrm{i}\theta_i}, 1\leq i\leq N$. Let $F_N(x)=\sharp\{i:  0\leq \theta_i\leq x\}/N, x\in [0,2\pi]$.  Since the eigenvalues of CUE are on a unit circle and all the eigenvalues shall be regarded as ``bulk" eigenvalues, we shall modify the definition of $\mathcal{A}_N$ to avoid the accumulation of the fluctuation of the eigenvalues around the origin. Hence, we choose  the statistic
\begin{align}
\mathcal{A}_N^{\text{CUE}}:=\int_{0}^{2\pi}\int_{0}^{2\pi} \Big(\big(F_N(y)-F_N(x)\big)-\frac{y-x}{2\pi }\Big)^2 {\rm d} x{\rm d} y. \label{19053002}
\end{align}
 One can also consider the Rains' statistic in \cite{Rains1997}, or the following  Watson's statistic  
\begin{align*}
\int \Big(F_N(x)-F(x)-\int(F_N(t)-F(t)){\rm d} F(t)\Big)^2 {\rm d}F(x) 
\end{align*}
which is independent of the choice of the origin; see \cite{Stephens76}. In the sequel, we discuss $\mathcal{A}_N^{\text{CUE}}$ in (\ref{19053002}) only. Our result is stated as the following theorem.
\begin{thm}\label{thm.CUE} Let $\mathcal{A}_N^{\text{CUE}}$ be as in (\ref{19053002}). Let $\{Y_i\}_{i=1}^\infty$ be a collection of  i.i.d.\  $N_{\mathbb{C}}(0,1)$,  i.e.\ complex standard Gaussian random variables with i.i.d. $N(0,\frac12)$ real and imaginary  parts. We have 
\begin{align*}
\frac{N^2}{4}\mathcal{A}_N^{\text{CUE}}-(\log N+\gamma+1)\Longrightarrow \sum_{j=1}^\infty \frac{1}{j}\big(|Y_j|^2-1\big), 
\end{align*}
where $\gamma$ is the Euler's constant.  
\end{thm}
The proof of Theorem \ref{thm.CUE} will be given at the end of this appendix,  after we introduce some necessary preliminary results. 

 Consider a square integrable function  $f: [0,2\pi]\to \mathbb{R}$. We denote the Fourier expansion of $f(\theta)$ by $\sum_{j\in\mathbb{Z}} \widehat{f}_j \mathrm{e}^{\ii  j\theta}$, where 
\begin{align*}
\widehat{f}_j:= \frac{1}{2\pi}\int_{0}^{2\pi} f(\theta) \mathrm{e}^{-\ii j\theta} {\rm d} \theta. 
\end{align*}
It is elementary to check  the following Fourier expansion of the indicator function 
\begin{align*}
\mathbf{1}_{[\alpha,\beta]}(\theta)\mapsto  \frac{1}{2\pi }(\beta-\alpha)+\frac{1}{2\pi \ii}\sum_{j=1}^\infty \frac{\mathrm{e}^{-\ii j\alpha}-\mathrm{e}^{-\ii j\beta}}{j} \mathrm{e}^{\ii j\theta}
+\frac{1}{2\pi \ii}\sum_{j=1}^\infty \frac{\mathrm{e}^{\ii j\beta}-\mathrm{e}^{\ii j\alpha}}{j} \mathrm{e}^{-\ii j\theta}, \qquad 0\leq \alpha<\beta\leq 2\pi.
\end{align*}
It is well known that the series on the RHS above converges to the indicator function pointwise except for only two points $\alpha$ and $\beta$. 
Therefore, for $x\leq y$, we have 
\begin{align*}
F_N(y)-F_N(x) - \frac{y-x}{2\pi }\mapsto \frac{1}{2N\pi \ii}\sum_{j=1}^\infty \frac{\mathrm{e}^{-\ii jx}-\mathrm{e}^{-\ii jy}}{j} \text{Tr}U^j
+\frac{1}{2N\pi \ii}\sum_{j=1}^\infty \frac{\mathrm{e}^{\ii jy}-\mathrm{e}^{\ii jx}}{j} \text{Tr}\overline{U}^j.
\end{align*}
Consequently, we have 
\begin{align*}
\mathcal{A}_N^{\text{CUE}}=&-\frac{1}{2N^2\pi^2}\int_{0}^{2\pi}\bigg(\int_{0}^y \Big(\sum_{j=1}^\infty \frac{\mathrm{e}^{-\ii jx}-\mathrm{e}^{-\ii jy}}{j} \text{Tr}U^j+\sum_{j=1}^\infty \frac{\mathrm{e}^{\ii jy}-\mathrm{e}^{\ii jx}}{j} \text{Tr}\overline{U}^j
\Big)^2{\rm d} x\bigg) {\rm d} y \nonumber\\
=
&\frac{4}{N^2}\sum_{j=1}^\infty \frac{1}{j^2}\text{Tr}U^j \text{Tr}\overline{U}^j= \frac{4}{N^2}\sum_{j=1}^N \frac{1}{j^2}\text{Tr}U^j \text{Tr}\overline{U}^j+\frac{4}{N^2}\sum_{j=N+1}^\infty \frac{1}{j^2}\text{Tr}U^j \text{Tr}\overline{U}^j
\end{align*}
To study the expectation and fluctuation of $\mathcal{A}_N^{\text{CUE}}$, we need the following result from \cite{DE01}.
\begin{thm}[Theorem 2.1, \cite{DE01}]\label{thm.022501} (a) Consider $a=(a_1,\cdots, a_k)$ and $b=(b_1, \cdots, b_k)$ with $a_j, b_j\in \{0,1,\cdots\}$. Let $Y_1, Y_2, \ldots, Y_k$ be independent standard complex normal random variables. Then $N\geq (\sum_{j=1}^k ja_j)\vee (\sum_{j=1}^k j b_j)$, 
\begin{align*}
\mathbb{E}\bigg( \prod_{j=1}^k (\mathrm{Tr} U^j )^{a_j}\overline{\big(\mathrm{Tr} U^j\big)^{b_j}}\bigg)= \mathbb{E} \bigg( \prod_{j=1}^k (\sqrt{j}Y_j)^{a_j}\overline{(\sqrt{j}Y_j)^{b_j}} \bigg)=\delta_{ab} \prod_{j=1}^k j^{a_j} a_j !
\end{align*} 
(b) For any $j,k$, 
\begin{align*}
\mathbb{E} \bigg( \mathrm{Tr} U^j\mathrm{Tr} \overline{U^k}\bigg)=\delta_{jk} (j\wedge N). 
\end{align*}
\end{thm}
With the above facts, we  now state our proof of Theorem \ref{thm.CUE}.
\begin{proof}[Proof of Theorem \ref{thm.CUE}]
First, based on Theorem \ref{thm.022501}, it is elementary to compute 
\begin{align}
\mathbb{E}\mathcal{A}_N^{\text{CUE}}= \frac{4}{N^2}\sum_{j=1}^N \frac{1}{j}+\frac{4}{N}\sum_{j=N+1}^\infty \frac{1}{j^2}= \frac{4}{N^2}(\log N+\gamma+1)+O\big(\frac{1}{N^3}\big).  \label{20070701}
\end{align}
Next, we need to identify the limiting distribution of $\mathcal{A}_N^{\text{CUE}}$ after centering. We write  
\begin{align} 
\mathcal{A}_N^{\text{CUE}}-\mathbb{E}\mathcal{A}_N^{\text{CUE}}= &\frac{4}{N^2}\sum_{j=1}^M \frac{1}{j^2}\big(\text{Tr}U^j \text{Tr}\overline{U}^j-j\big)+\frac{4}{N^2}\sum_{j=M+1}^N \frac{1}{j^2}\big(\text{Tr}U^j \text{Tr}\overline{U}^j-j\big)\nonumber\\
&+\frac{4}{N^2}\sum_{j=N+1}^\infty \frac{1}{j^2}\big(\text{Tr}U^j \text{Tr}\overline{U}^j-N\big),\label{19022501}
\end{align}
where we can choose $M$ to be sufficiently large. Our aim is to show that the last two terms in (\ref{19022501}) are negligible in probability when $M$ is large. It would be sufficient to show that their second moments are negligible.  To this end, we first recall the jpdf of the (unordered) eigenvalues $\{\mathrm{e}^{\ii \theta_j}\}_{j=1}^N$ of CUE (see e.g.\, \cite{Forrester})
\begin{align}
p_N(\mathrm{e}^{\ii \theta_1}, \ldots, \mathrm{e}^{\ii \theta_N})= \frac{1}{N!(2\pi)^N} \prod_{1\leq i<j\leq N}\big|\mathrm{e}^{\ii \theta_i}-\mathrm{e}^{\ii \theta_j}\big|^2, \qquad \theta_i\in [0,2\pi], \quad i=1,\ldots, N.  \label{jpdf of CUE}
\end{align}
It is also well known that $\{\mathrm{e}^{\ii \theta_j}\}_{j=1}^N$ is a determinantal point process with kernel (see e.g., \cite{Forrester})
\begin{align}
K_N(\theta,\theta')= \frac{1}{2\pi}\mathrm{e}^{\ii\frac{(N-1)(\theta-\theta')}{2}}\frac{\sin \frac{N(\theta-\theta')}{2}}{\sin \frac{\theta-\theta'}{2}}=\frac{1}{2\pi}\sum_{k=0}^{N-1}\mathrm{e}^{\ii k(\theta-\theta')}.  \label{20070710}
\end{align}
More specifically, for any $n\in\llbracket 1, N \rrbracket$, the $n$-point correlation function of the point process $\{\mathrm{e}^{\ii \theta_j}\}_{j=1}^N$  can be  written as 
\begin{align*}
p_N^{(n)}(\mathrm{e}^{\ii \theta_1}, \ldots, \mathrm{e}^{\ii \theta_n})&:= \int_{[0,2\pi]^{N-n}} p_N(\mathrm{e}^{\ii \theta_1}, \ldots, \mathrm{e}^{\ii \theta_N}) {\rm d} \theta_{n+1}\cdots {\rm d}\theta_N\nonumber\\
&= \frac{(N-n)!}{N!} \det (K_N(\theta_i,\theta_j))_{1\leq i,j\leq n}, \qquad \theta_i\in [0,2\pi], \quad i=1,\ldots, n. 
\end{align*}
Apparently, in order to compute $\text{Cov}\Big(\text{Tr} U^j\text{Tr} \overline{U}^j,  \text{Tr} U^k \text{Tr} \overline{U}^k\Big)$, it suffices to use the formula of $p_N^{(n)}$ with $n\leq 4$. Applying the explicit formula (\ref{20070710}), it is  elementary 
to check 
\begin{align}
&\text{Cov}\Big(\text{Tr} U^j\text{Tr} \overline{U}^j,  \text{Tr} U^k \text{Tr} \overline{U}^k\Big)\nonumber\\
&=
\left\{\begin{array}{llll}
j^2\delta_{jk},  & \text{if} \quad 1\leq j+k\leq N \\
j^2\delta_{jk}+N-j-k,  & \text{if} \quad j+k\geq N+1,\quad j,k\leq N\\
N^2\delta_{jk}-(N-|k-j|)\vee 0, & \text{if} \; j \;\text{or} \; k\geq N.
\end{array}
\right. \label{19100605}
\end{align} 

Then,  using Theorem \ref{thm.022501}, we can first conclude that $\{\frac{1}{\sqrt{j}}\text{Tr} U^j\}_{j=1}^M$ converges jointly to i.i.d. $N_{\mathbb{C}}(0,1)$ variables  for any large fixed $M$. Further, using (\ref{19100605}) to the second and the third parts  in (\ref{19022501}), one can easily check that these two part are negligible (in $M$) in probability when $M$ is large. Similarly to the discussions in \eqref{3.13} -- \eqref{3.15}, we can conclude that  
\begin{align}
\frac{N^2}{4}\Big(\mathcal{A}_N^{\text{CUE}}-\mathbb{E}\mathcal{A}_N^{\text{CUE}} \Big)\Longrightarrow \sum_{j=1}^\infty \frac{1}{j}\big(|Y_j|^2-1\big), \label{20070301}
\end{align}
where $Y_i$'s are i.i.d. $N_{\mathbb{C}}(0,1)$.   Combining (\ref{20070701}) with (\ref{20070301}) we can complete the proof. 
\end{proof}

\end{document}